\documentclass[12pt]{amsart}

\usepackage[utf8]{inputenc}
\usepackage[colorlinks,anchorcolor=blue,citecolor=blue,linkcolor=blue,urlcolor=blue,bookmarksdepth=2]{hyperref}

\usepackage{mathdots}
\usepackage{diagbox}

\usepackage{amssymb, amscd, txfonts}
\usepackage{graphicx}
\usepackage[all]{xy} 
\usepackage{array, booktabs}
\usepackage{tabularx} 
\usepackage{appendix}

\numberwithin{equation}{section}

\newtheorem{theorem}{Theorem}[section]
\newtheorem{proposition}[theorem]{Proposition}
\newtheorem{lemma}[theorem]{Lemma}
\newtheorem{corollary}[theorem]{Corollary}

\theoremstyle{definition}
\newtheorem{definition}[theorem]{Definition}

\theoremstyle{remark}
\newtheorem{remark}[theorem]{Remark}
\newtheorem{claim}[theorem]{Claim}
\newtheorem{step}{Step}

\newcommand{\Z}{\mathbb{Z}}
\newcommand{\Q}{\mathbb{Q}}
\newcommand{\R}{\mathbb{R}}
\newcommand{\C}{\mathbb{C}}
\newcommand{\Zp}{\mathbb{Z}_{p}}
\newcommand{\Qp}{\mathbb{Q}_{p}}
\newcommand{\A}{\mathbb{A}}
\newcommand{\Af}{\mathbb{A}_{{\rm fin}}}
\newcommand{\SO}{{\rm SO}}

\newcommand{\bbA}{\mathbb{A}}

\newcommand{\bbC}{\mathbb{C}}

\newcommand{\bbQ}{\mathbb{Q}}

\newcommand{\bbR}{\mathbb{R}}

\newcommand{\bbZ}{\mathbb{Z}}

\newcommand{\calF}{\mathcal{F}}

\newcommand{\calS}{\mathcal{S}}

\newcommand{\frakg}{\mathfrak{g}}

\newcommand{\frakk}{\mathfrak{k}}

\newcommand{\frakl}{\mathfrak{l}}

\newcommand{\frakp}{\mathfrak{p}}

\newcommand{\frakq}{\mathfrak{q}}

\newcommand{\frakS}{\mathfrak{S}}

\newcommand{\fraku}{\mathfrak{u}}

\newcommand{\GL}{\mathrm{GL}}

\newcommand{\SL}{\mathrm{SL}}

\newcommand{\Sp}{\mathrm{Sp}}

\newcommand{\PGSp}{\mathrm{PGSp}}
\newcommand{\Mp}{\mathrm{Mp}}

\newcommand{\Sym}{\mathrm{Sym}}

\newcommand{\diag}{\mathrm{diag}}
\newcommand{\sgn}{\mathrm{sgn}}

\newcommand{\vep}{\varepsilon}
\newcommand{\isom}{\cong}

\newcommand{\Irr}{\mathrm{Irr}}

\begin{document}

\title[]{Holomorphic differential forms on some orthogonal modular varieties}
\author[]{Shuji Horinaga}
\address{NTT Institute for Fundamental Mathematics, NTT Inc., Japan}
\email{syuuji.horinaga@ntt.com, shorinaga@gmail.com}
\author[]{Shouhei Ma}
\address{Department~of~Mathematics, Institute~of~Science~Tokyo, Tokyo 152-8551, Japan}
\email{ma@math.titech.ac.jp}
\thanks{Supported by JSPS KAKENHI 23K12965 and 21H00971.}

\begin{abstract}
We construct holomorphic differential forms of many degrees, 
including the minimum possible one, 
on the modular varieties associated to the even lattices 
of signature $(2, n)$ with $n\equiv 1, 3$ mod $8$ and discriminant $-2$ 
in the range $n\geq 25$. 
This is the first example of holomorphic differential forms of non-top degree 
on orthogonal modular varieties. 
The proof uses the Arthur multiplicity formula 
in the theory of automorphic representations. 
\end{abstract} 

\maketitle

\section{Introduction}\label{sec: intro}

Historically, 
the study of the birational type of higher dimensional modular varieties 
started with construction of holomorphic differential forms of non-top degree \cite{Fr}. 
Soon after, with the development of the theory of toroidal compactifications, 
the focus shifted to pluricanonical forms and especially their growth, the Kodaira dimension. 
While holomorphic differential forms are concerned with vector-valued modular forms,  
pluricanonical forms are concerned with scalar-valued ones, which are in general easier to construct. 
Nowadays, Kodaira dimension is regarded as the primary invariant in the birational classification. 
However, if one wants to look into finer geometry in the next stage 
such as Hodge structures and subvarieties, 
one needs to proceed (or go back) to holomorphic differential forms 
as the next set of fundamental birational invariants. 

In the case of orthogonal modular varieties, 
we now know much about the Kodaira dimension in higher dimension 
(\cite{GHS}, \cite{Ma1}), 
while we know nothing about holomorphic differential forms except for 
a basic vanishing theorem in low degree (\cite{Ma2}). 
The purpose of this paper is to construct holomorphic differential forms on 
a special class of orthogonal modular varieties,  
with the aid of the Arthur multiplicity formula in the theory of automorphic representations. 
At present this technique is available only for unimodular-like lattices 
(as is usual with applications of the Arthur theory), 
but we tried to establish the framework of application 
in a wider generality. 

Let $L$ be a lattice of signature $(2, n)$. 
Let $\mathcal{D}$ be the Hermitian symmetric domain attached to $L$, 
and ${\rm O}^{+}(L)$ be the subgroup of 
the orthogonal group ${\rm O}(L)$ preserving $\mathcal{D}$.  
The quotient 
${\rm O}^{+}(L)\backslash \mathcal{D}$ 
has the structure of a quasi-projective variety of dimension $n$, 
known as an \textit{orthogonal modular variety}. 
For some reasons coming from automorphic representations 
(explained later), 
we assume $n\equiv 1, 3 \bmod 8$ and consider the following special even lattices: 
\begin{equation}\label{eqn: lattice intro}
L = 
\begin{cases}
2U \oplus m E_8 \oplus A_1 & \; n\equiv 3 \bmod 8, \\ 
2U \oplus m E_8 \oplus E_7 & \; n\equiv 1 \bmod 8. 
\end{cases}
\end{equation}
Here $U$ is the integral hyperbolic plane, 
and the root lattices are the negative-definite ones. 
Thus $L$ is the even lattice of minimal discriminant in each $n$. 
It is not unimodular, but has the special property that 
${\SO}(L\otimes {\Zp})$ is a hyperspecial subgroup of 
${\SO}(L\otimes {\Qp})$ for every prime $p<\infty$. 
With this choice of lattice, we shall write simply 
\begin{equation*}
\mathcal{F}_{n} = {\rm O}^{+}(L) \backslash \mathcal{D}. 
\end{equation*}
Our goal is to construct holomorphic differential forms 
on a smooth projective model of $\mathcal{F}_{n}$. 
It is known that there is no nonzero holomorphic $k$-form 
in the range $k<n/2$ (see, e.g., \cite{Ma2} Chapter 9). 
In what follows, we consider $n/2< k <n$. 

Our main result is the following. 

\begin{theorem}\label{thm: main intro}
Let $n\equiv 1, 3$ mod $8$ with $n\geq 25$. 
There exists a nonzero holomorphic $k$-form 
on a smooth projective model of $\mathcal{F}_{n}$ for 

(i) every odd $k$ in the range described in 
Table \ref{table: range k odd}; 

(ii) every even $k$ in the range described in 
Table \ref{table: range k even}; 

(iii) for $25 \leq n \leq 35$, the values of $k$ in 
Table \ref{table: small n intro}. 
\end{theorem}

\renewcommand{\arraystretch}{1.2}
\begin{table}
    \centering
    \caption{Odd $k$}
    \label{table: range k odd}
    \begin{tabular}{c|c|c|c}
    $n\bmod 8$ & $k \bmod 4$ & range of $k$ & bound of $n$ \\
    \hline
    $1$ & $1$ & $(n+1)/2 \leq k <n$ & $n\geq 33$  \\
    \hline
    $3$ & $1$ & $(n+7)/2 \leq k <n$ & $n\geq 35$  \\
    \hline
    $1, 3$ & $3$ & $(3n+25)/4 \leq k <n$ & $n\geq 33$ 
    \end{tabular}
\end{table}

\begin{table}
    \centering
    \caption{Even $k$}
    \label{table: range k even}
    \begin{tabular}{c|c|c|c}
    $n\bmod 8$ & $k \bmod 8$ & range of $k$ & bound of $n$ \\
    \hline
    $1, 3$ & $6$ & $(4n+4)/5 \leq k <n$ & $n\geq 25$  \\
    \hline
    $1, 3$ & $0$ & $(6n+8)/7 < k <n$ & $n\geq 33$ \\
    \hline
    $1, 3$ & $4$ & $(6n+18)/7 \leq k <n$ & $n\geq 57$ \\
    \hline
    $1$ & $2$ &  $(4n+24)/5 \leq k <n$ & $n\geq 65$ \\
    \hline
     $3$ & $2$ &  $(8n+12)/9 \leq k <n$ & $n\geq 27$ 
    \end{tabular}
\end{table}

\begin{table}
    \centering
    \caption{Small $n$}
    \label{table: small n intro}
    \begin{tabular}{c|c}
    $n$ & $k$ \\
    \hline
    $25$ & $17, 22$  \\
    \hline
    $27$ & $17, 23, 25, 26$   \\
    \hline
    $33$ &  $17, 21, 25, 27, 29, 30, 31, 32$  \\
    \hline
    $35$ &  $21, 25, 27, 29, 30, 31, 33, 34$  
    \end{tabular}
\end{table}

In the cases (i) and (ii), we intend to attain many values of $k$ 
in systematic ways for sufficiently large $n$. 
In general, it seems more easy to attain odd $k$ than even $k$. 
Note that the minimal possible value $k=(n+1)/2$ is attained 
in the case $(n, k)\equiv (1, 1)$. 
This shows that the vanishing bound $k<n/2$ is optimal 
as a general bound. 
Furthermore, in that case, 
the space of square-integrable $(n+1)/2$-forms 
is isomorphic to the space of cusp forms of weight $(n+3)/2$ 
for ${\SL}_2(\Z)$ (Proposition \ref{cor: exact isom}). 
The case (iii) supplements (i) and (ii) by giving 
an explicit and improved list for small values of $n$.

Theorem \ref{thm: main intro} is the first example of 
holomorphic differential forms of non-top degree 
on orthogonal modular varieties. 
The distribution of $k$ looks rather free 
at least when $k$ is odd or close to $n$.  
This behavior seems different from the Siegel modular case, 
where there is a strong constraint on 
the possible values of $k$ (\cite{We}). 

The bound $n\geq 25$ in Theorem \ref{thm: main intro} is optimal.  
Indeed, $\mathcal{F}_{19}$ is the moduli space of 
$K3$ surfaces of degree $2$, which is unirational. 
Similarly, moduli interpretation via $K3$ surfaces 
tells us that 
$\mathcal{F}_{17}$, $\mathcal{F}_{11}$ and $\mathcal{F}_{9}$ 
are rational (\cite{Ma0}). 
In our view, this is a place where 
algebraic geometry and automorphic representations 
affect each other indirectly. 
Algebraic geometry produces unirational modular varieties 
in relatively small dimension, 
and this gives a lower bound (of weight) where 
automorphic construction ceases to work. 
Conversely, automorphic technique produces 
holomorphic forms in higher dimension, 
which serve as obstructions for finding unirational moduli spaces. 

In the case $n\equiv 3 \bmod 8$, 
it is proved in \cite{GHS} that 
$\mathcal{F}_{n}$ is of general type in the range $n\geq 43$. 
Similarly, in the case $n\equiv 1 \bmod 8$, we can show that 
$\mathcal{F}_{n}$ is of general type if $n=25$ or $n\geq 41$ 
(Proposition \ref{prop: Kodaira dim}). 
It appears, at least apparently, that 
these methods do not cover the remaining cases $n= 27, 33, 35$. 
Theorem \ref{thm: main intro} tells us an information
in these cases: 

\begin{corollary}
    $\mathcal{F}_{27}$, $\mathcal{F}_{33}$ and $\mathcal{F}_{35}$ 
    are not unirational. 
\end{corollary}

In the case $n=27$, 
the lattice $L=2U\oplus 3E_8\oplus A_1$ contains 
the even unimodular lattice $II_{2,26}=2U\oplus 3E_8$ as a sublattice.  
Accordingly, $\mathcal{F}_{27}$ contains the modular variety 
$\mathcal{F}_{26}$ for ${\rm O}^{+}(II_{2,26})$ as a divisor. 
Then, by restricting 
the holomorphic forms on $\mathcal{F}_{27}$ constructed in 
Theorem \ref{thm: main intro} (iii) to this divisor, 
we obtain holomorphic forms on $\mathcal{F}_{26}$, 
though possibly zero. 
This is the first potential construction of 
a holomorphic tensor on $\mathcal{F}_{26}$ 
whose birational type remains mysterious. 
One possible approach for proving nonvanishing of restriction 
(if one believes so) would be to consider the chain 
$\mathcal{F}_{25} \subset \mathcal{F}_{26} \subset \mathcal{F}_{27}$ 
and show that the $17$-form on 
$\mathcal{F}_{27}$ restricts to the one on $\mathcal{F}_{25}$.

The proof of Theorem \ref{thm: main intro} uses the theory of automorphic representations. 
Holomorphic differential forms on a smooth projective model of $\mathcal{F}_{n}$ 
can be identified with certain vector-valued modular forms for 
${\SO}^{+}(L)$ (\cite{Ma2}). 
By a standard procedure, the problem of constructing such a modular form 
can be translated to finding a certain type of automorphic representation of 
${\SO}(L\otimes {\A})$ with prescribed archimedean component. 

Arthur's multiplicity formula 
(\cite{Arthur}, \cite{AGI+}, \cite{Taibi_AMF}, \cite{2024_Ishimoto_Arthur}) 
is a powerful tool for this type of problem. 
Roughly speaking, for our purpose, 
it provides a lifting from a certain collection of automorphic representations of 
${\rm GL}_{N}({\A})$ with $N$ small 
to an automorphic representation of 
${\SO}(L\otimes {\A})$. 
For our input, we will take $1\leq N \leq 5$. 
The cases $N=2, 3$ will come from elliptic modular forms, 
and the cases $N=4, 5$ from vector-valued Siegel modular forms of genus $2$. 
Thus our actual inputs are elliptic modular forms and Siegel modular forms of genus $2$.

In general, inputs of the Arthur multiplicity formula are 
required to satisfy a certain character identity. 
This takes the form of a product formula over $p\leq \infty$ 
analogous to the Hilbert reciprocity. 
In our case, 
we take the local character at $p<\infty$ to be trivial. 
This reduces the character identity to an equality 
at the archimedean place (see Proposition \ref{prop: criterion}), 
which is eventually turned to a combinatorial puzzle 
among classical cusp forms as above. 

At this point we can explain our choice of the lattices. 
First, when the lattice $L$ has odd rank, 
we can pass from 
${\rm O}^{+}(L)=\langle {\SO}^{+}(L), -{\rm id} \rangle$ 
to ${\SO}^{+}(L)$ without changing the modular variety, 
and also the relevant representation theory is somewhat simpler 
than the even rank case. 
Secondly, the condition $n\equiv 1, 3 \bmod 8$ comes from 
the requirement that 
${\SO}(L\otimes{\Qp})$ is split ($\Leftrightarrow$ quasi-split) 
at every $p<\infty$ 
(Proposition \ref{prop: everywhere split}). 
Finally, we need ${\SO}(L\otimes{\Zp})$ to be hyperspecial 
at every $p<\infty$, 
due to the current lack of local newform theory. 
This forces us to work with the unimodular-like lattices 
\eqref{eqn: lattice intro}. 

Application of the Arthur multiplicity formula to the birational type of 
modular varieties was done by 
Maeda, Yamauchi and the first-named author \cite{HMY} 
in the case of ball quotients. 
While they constructed pluricanonical forms, 
our target is holomorphic differential forms. 
The latter has in a sense smaller weight than canonical forms, 
but instead we have the Pommerening extension theorem \cite{Po}. 
Finally, it should be noted that, 
since the Arthur multiplicity formula 
is conditional on the weighted twisted fundamental lemma, 
so is our result. 

This paper is organized as follows.
In \S \ref{sec: auto rep to hol form} and \S \ref{sec: adelize}, 
we reduce the construction of holomorphic differential forms
to finding certain adelic automorphic representations. 
This is done with general lattices. 
In \S \ref{sec: lattice}, 
we specialize our lattices to \eqref{eqn: lattice intro}. 
In \S \ref{sec: AMT}, we recall the Arthur multiplicity formula.
\S \ref{sec: AJ} is devoted to preliminary calculations 
at the archimedean place. 
In \S \ref{sec: A-packet}, we construct the desired $A$-parameters.
This completes the proof of Theorem 1.1.
In \S \ref{sec: everywhere split}, 
we classify rational quadratic forms of signature $(2, n)$ 
whose orthogonal group is split at every $p<\infty$. 
In \S \ref{sec: Kodaira dim}, we supply a general-type result 
in the case $n\equiv 1 \bmod 8$.

Throughout this paper, 
a \textit{lattice} means a free $\Z$-module $L$ of finite rank 
equipped with a nondegenerate symmetric bilinear form 
$(\: , \: ) : L\times L\ \to \Z$.  
The symbol $L\otimes F=L\otimes_{{\Z}}F$ 
stands for the quadratic space over a field $F$ obtained 
from the lattice $L$ by extension of scalars. 
Restricted direct product over $p\leq \infty$ (or $p<\infty$) 
will be simply denoted by $\prod_{p}$. 
Similarly, restricted tensor product will be 
simply denoted by $\otimes_{p}$.

\section{Holomorphic differential forms and lowest weight modules}\label{sec: auto rep to hol form}

In \S \ref{sec: auto rep to hol form} and \S \ref{sec: adelize}, 
we translate the problem of constructing a holomorphic $k$-form on 
our orthogonal modular variety to 
finding a certain type of automorphic representation of ${\SO}(L\otimes {\A})$. 
This \S \ref{sec: auto rep to hol form} is devoted to 
the passage to  
the Lie group ${\SO}(L\otimes {\R})$, 
and the next \S \ref{sec: adelize} is devoted to 
the adelization. 

Let $L$ be a lattice of signature $(2, n)$ with $n\geq 3$. 
We write ${\SO}^{+}(L\otimes {\R})$ for the identity component of ${\SO}(L\otimes{\R})$, 
and let $K_{\infty}^{+}\simeq {\SO}(n, {\R})\times {\SO}(2, {\R})$ 
be a maximal compact subgroup of ${\SO}^{+}(L\otimes {\R})$. 
We also write $${\SO}^{+}(L)={\SO}(L)\cap {\SO}^{+}(L\otimes {\R}).$$ 
Referring to the following subsections 
for the notation and terminology, 
the result of this section can be stated as follows.

\begin{proposition}\label{prop: hol form via auto rep}
Let $\Gamma$ be a finite-index subgroup of ${\SO}^{+}(L)$. 
Suppose that the $K_{\infty}^{+}$-finite part of 
$L^{2}(\Gamma \backslash {\SO}^{+}(L\otimes {\R}))$ 
contains a lowest weight module 
of weight $(\wedge^{n-k}, -k)$ for some $n/2<k<n$. 
Then a smooth projective model of the modular variety 
$\Gamma \backslash \mathcal{D}$
has a nonzero holomorphic $k$-form. 
\end{proposition}

This section is devoted to explaining and proving Proposition \ref{prop: hol form via auto rep}. 
The process is divided into three steps, to which our subsections correspond: 
\begin{enumerate}
\item from modular forms to holomorphic differential forms (\S \ref{ssec: VOMF})
\item from automorphic forms on ${\SO}^{+}(L\otimes {\R})$ to modular forms (\S \ref{ssec: lift to G})
\item from lowest weight modules to automorphic forms (\S \ref{ssec: LWM})
\end{enumerate}

In the next \S \ref{sec: adelize}, 
we will specialize $\Gamma$ to ${\SO}^{+}(L)$. 

\subsection{Vector-valued modular forms}\label{ssec: VOMF}

In this subsection we prove that 
holomorphic differential forms on a smooth projective model of 
an orthogonal modular variety can be identified with certain vector-valued modular forms. 
We refer to \cite{Ma2} for the basic theory of vector-valued modular forms, 
but since we will be concerned only with a special type of weights, 
we do not need to recall the full theory in \cite{Ma2}. 

Let $L$ be a lattice of signature $(2, n)$ with $n\geq 3$. 
The Hermitian symmetric domain $\mathcal{D}=\mathcal{D}_{L}$ attached to $L$ 
is defined as a connected component of the space 
\begin{equation*}
\{ \; [\omega]\in \mathbb{P}(L\otimes {\C}) \: | \: (\omega, \omega)=0, \, (\omega, \bar{\omega})>0 \; \}. 
\end{equation*}
For a subgroup $\Gamma <{\SO}^{+}(L)$ of finite index, the quotient 
$\Gamma \backslash \mathcal{D}$ has the structure of a quasi-projective variety of dimension $n$. 
It is known that the smooth locus of $\Gamma \backslash \mathcal{D}$ 
has no nonzero holomorphic $k$-form in the range $0<k<n/2$ (see \cite{Ma2} Chapter 9). 

Let $\mathcal{L}=\mathcal{O}(-1)|_{\mathcal{D}}$ be the Hodge line bundle on $\mathcal{D}$, 
and $\mathcal{E}=\mathcal{L}^{\perp}/\mathcal{L}$ be the second Hodge bundle 
where $\mathcal{L}^{\perp}$ is the orthogonal complement of $\mathcal{L}$ as a sub line bundle of $L\otimes \mathcal{O}_{\mathcal{D}}$. 
Then $\mathcal{L}$ and $\mathcal{E}$ are homogeneous vector bundles associated to 
the standard representations of ${\SO}(2, {\R})$ and ${\SO}(n, {\R})$ respectively. 

Let $k, \ell >0$ be natural numbers with $\ell < n/2$. 
A $\Gamma$-invariant holomorphic section of the automorphic vector bundle 
$\wedge^{\ell}\mathcal{E}\otimes \mathcal{L}^{\otimes k}$ over $\mathcal{D}$ 
is called a \textit{modular form} of weight $(\wedge^{\ell}, k)$ with respect to $\Gamma$. 
We denote by $M_{\wedge^{\ell}, k}(\Gamma)$ the space of such modular forms.

\begin{proposition}\label{prop: MF to hol form}
Let $n/2 < k < n$. 
Let $X$ be a smooth projective model of $\Gamma \backslash \mathcal{D}$. 
Then we have a natural isomorphism 
$H^0(X, \Omega_{X}^{k})\simeq M_{\wedge^{n-k}, k}(\Gamma)$. 
\end{proposition}

\begin{proof}
We first consider the case $\Gamma$ torsion-free. 
Then we have 
\begin{equation*}
\wedge^{n-k}\mathcal{E} \simeq (\wedge^{k}\mathcal{E})^{\vee}\otimes \det \mathcal{E} 
\simeq (\wedge^{k}\mathcal{E})^{\vee} \simeq \wedge^{k}\mathcal{E} 
\end{equation*}
as $\Gamma$-equivariant vector bundles. 
Here we used $\Gamma < {\SO}^{+}(L)$ and the self-duality of $\wedge^{k}$ as an ${\SO}$-representation. 
Then we have 
\begin{equation*}
\Omega_{\mathcal{D}}^{k} \simeq \wedge^{k}\mathcal{E}\otimes \mathcal{L}^{\otimes k} 
\simeq \wedge^{n-k}\mathcal{E}\otimes \mathcal{L}^{\otimes k}. 
\end{equation*}
(See \cite{Ma2} Example 2.3 for the first isomorphism.) 
This shows that 
\begin{equation*}
M_{\wedge^{n-k}, k}(\Gamma) \simeq H^0(\mathcal{D}, \Omega_{\mathcal{D}}^{k})^{\Gamma} 
\simeq H^0(\Gamma \backslash \mathcal{D}, \Omega^{k}). 
\end{equation*}
By the extension theorem of Pommerening \cite{Po}, 
every holomorphic $k$-form on 
$\Gamma \backslash \mathcal{D}$ 
extends holomorphically over a smooth projective compactification $X$ of $\Gamma \backslash \mathcal{D}$. 
Thus we have 
$H^0(\Gamma \backslash \mathcal{D}, \Omega^{k}) = H^0(X, \Omega_{X}^{k})$. 

For general $\Gamma$, we choose a torsion-free normal subgroup $\Gamma' \lhd \Gamma$ of finite index. 
Let $X'$ be a smooth projective model of $\Gamma' \backslash \mathcal{D}$. 
We take the $\Gamma/\Gamma'$-invariant part of the isomorphism 
$H^0(X', \Omega_{X'}^{k})\simeq M_{\wedge^{n-k}, k}(\Gamma')$ 
for $\Gamma'$. 
Then 
$M_{\wedge^{n-k}, k}(\Gamma')^{\Gamma/\Gamma'}=M_{\wedge^{n-k}, k}(\Gamma)$ 
by the definition of modular forms, 
while we have 
$H^0(X', \Omega_{X'}^{k})^{\Gamma/\Gamma'}=H^0(X, \Omega_{X}^{k})$ 
as a well-known property of holomorphic differential forms. 
\end{proof}

\subsection{Automorphic forms on the Lie group}\label{ssec: lift to G} 

We choose a base point $[\omega_{0}]$ of $\mathcal{D}$ 
and let $K_{\infty}^{+}$ be the stabilizer of $[\omega_{0}]$ in ${\SO}^{+}(L\otimes {\R})$. 
Then $K_{\infty}^{+}$ is isomorphic to ${\SO}(n, {\R})\times {\SO}(2, {\R})$. 
We denote by $\frakg =\frakk \oplus \frakp$ the associated Cartan decomposition of the Lie algebra of 
${\SO}^{+}(L\otimes \R)$. 
Let $\mathfrak{p}_{{\C}} = \mathfrak{p}_{+}\oplus \mathfrak{p}_{-}$ 
be the eigendecomposition for the adjoint action of the center $\mathfrak{so}(2, {\R})$ of $\mathfrak{k}$. 
Then $\mathfrak{p}_{-}$ gives the Cauchy-Riemann operator at each point of $\mathcal{D}$, 
while $\mathfrak{p}_{+}$ is identified with the complex tangent space. 

Let $f$ be a modular form of weight $(\wedge^{\ell}, k)$ for $\Gamma$. 
We define a $\Gamma$-invariant function $\tilde{f}$ on ${\SO}^{+}(L\otimes {\R})$ 
as follows (cf.~\cite{Ma2} \S 11.1). 
We first choose an isotropic line $I$ in $L\otimes {\Q}$. 
(This is fixed and subsumed in what follows.) 
The effect of this choice is to trivialize the vector bundle 
$\wedge^{\ell}\mathcal{E}\otimes \mathcal{L}^{\otimes k}$. 
Via this, we can identify $f$ with a vector-valued holomorphic function on $\mathcal{D}$ (denoted again by $f$) 
satisfying the invariance 
\begin{equation}\label{eqn: factor of automorphy}
f(\gamma [\omega]) = j(\gamma, [\omega]) f([\omega]), \qquad  \gamma \in \Gamma, \; [\omega]\in \mathcal{D}, 
\end{equation} 
with respect to the factor of automorphy $j$. 
Here the vector space where $f$ takes values is 
\begin{equation*}
W = (\wedge^{\ell}(I^{\perp}/I)\otimes (I^{\vee})^{\otimes k}) \otimes {\C}, 
\end{equation*}
and the factor of automorphy is a ${\rm GL}(W)$-valued function. 
Then we define a $W$-valued smooth function on ${\SO}^{+}(L\otimes {\R})$ by 
\begin{equation}\label{eqn: lift to G}
\tilde{f}(g) = j(g, [\omega_{0}])^{-1} f(g[\omega_{0}]), 
\qquad g\in {\SO}^{+}(L\otimes {\R}). 
\end{equation}
We regard $W$ as a $K_\infty^{+}$-representation via the trivialization over $[\omega_{0}]$ 
and denote by 
\begin{equation*}
\rho(h)=j(h, [\omega_{0}]), 
\qquad h\in K_\infty^{+}, 
\end{equation*}
this representation. 
This is irreducible with highest weight $(\wedge^{\ell}, k)$, and so isomorphic to 
\begin{equation*}
W_{\ell, k} = \wedge^{\ell}{\C}^n \boxtimes \chi^{k},  
\end{equation*} 
where $\chi^{k}$ is the weight $k$ character of ${\SO}(2, {\R})$.

 \begin{lemma}\label{lem: lift to G}
 The construction $f\mapsto \phi = \tilde{f}$ 
 gives an isomorphism between 
 the space of square-integrable modular forms of weight $(\wedge^{\ell}, k)$ for $\Gamma$ 
 and 
 the space of smooth $W$-valued functions $\phi$ on ${\SO}^{+}(L\otimes {\R})$ such that 
 
(1) $\phi(\gamma g)= \phi(g)$ for $\gamma\in \Gamma$;  
 
(2) $\phi(gh) = \rho(h)^{-1}(\phi(g))$ for $h\in K_\infty^{+}$; 
 
(3) $\mathfrak{p}_{-}\cdot \phi =0$; 
 
(4) $\phi$ is square-integrable over $\Gamma\backslash {\SO}^{+}(L\otimes {\R})$. 
\end{lemma}

\begin{proof}
The fact that the lifting $\tilde{f}$ defined by \eqref{eqn: factor of automorphy} 
satisfies (1) -- (4) is verified in \cite{Ma2} Claim 11.3. 
Conversely, given a $W$-valued function $\phi$ on 
${\SO}^{+}(L\otimes \R)$ satisfying (1) -- (4), 
we can define a $W$-valued function $f$ on $\mathcal{D}$ by 
$f(g[\omega_{0}]) =  j(g, [\omega_{0}])\phi(g)$. 
This is well-defined by (2). 
The property (3) is translated back to the holomorphicity of $f$, 
the property (1) to the $\Gamma$-invariance of $f$ in the sense of \eqref{eqn: factor of automorphy}, 
and the property (4) to the square-integrability of $f$. 
\end{proof}

\subsection{Lowest weight modules}\label{ssec: LWM}

In this subsection,  following \cite{Ma2} \S 11.2,  
we construct a lowest weight module from a modular form and vice versa. 
Let $0< \ell <n/2$ and $k>0$. 
A $(\mathfrak{g}, K_{\infty}^{+})$-module $M$ is called a 
\textit{lowest weight module} of weight 
$(\wedge^{\ell}, -k)$ 
if it is generated from a weight vector $v_0$ of weight 
\begin{equation}\label{eqn: lowest weight I}
(\wedge^{\ell}, k)^{\vee} = (\wedge^{\ell}, -k) 
= (\underbrace{1,\cdots, 1}_{\ell}, \underbrace{0,\cdots,0}_{[n/2]-\ell}, -k) 
\end{equation}
which is annihilated by $\mathfrak{p}_{-}$ and the positive root vectors of $\mathfrak{k}$. 
This is a highest weight module 
of weight $(\wedge^{\ell}, -k)$
in the sense of \cite{Hu} Chapter 1 
if we declare the (noncompact) roots for $\mathfrak{p}_{-}$ 
to be positive. 
The terminology ``lowest'' comes from the switch of the role of $\mathfrak{p}_{+}$ and $\mathfrak{p}_{-}$ 
(cf.~\eqref{eqn: root p+}). 
The $K_{\infty}^{+}$-representation $W_0$ generated by $v_0$ has 
highest weight \eqref{eqn: lowest weight I}, 
and hence is isomorphic to $W_{\ell, k}^{\vee}$.  
The module $M$ is generated from $W_0$ by the action of the universal enveloping algebra of $\frakp_+$. 

In general, the lowest weight module $M$ 
lies between the parabolic Verma module of weight 
$(\wedge^{\ell}, -k)$ (in the sense of \cite{Hu} \S 9.4) 
and its unique irreducible quotient. 
We denote the latter as $L(\wedge^{\ell}, -k)$.  
When $\ell+k>n$, the parabolic Verma module 
is already irreducible (\cite{EHW_83}), 
so we have $M=L(\wedge^{\ell}, -k)$. 
(The case we are interested in is $\ell+k=n$, 
but see Remark \ref{rmk: irr LWM}.) 

We go back to modular forms. 
Let $f$ be a square-integrable modular form of weight $(\wedge^{\ell}, k)$ 
and $\tilde{f}$ be its lift defined in \eqref{eqn: lift to G}. 
We choose a linear function $\alpha \colon W\to {\C}$, 
and let $V_{f}\subset L^2(\Gamma\backslash {\SO}^+(L\otimes {\R}))$ 
be the Hilbert subspace generated by the right translations of 
the function $\alpha \circ \tilde{f}$. 
This is independent of the choice of $\alpha$. 
Let $(V_{f})_{K_\infty^{+}}$ be 
the $K_\infty^{+}$-finite part of $V_{f}$. 

\begin{proposition}\label{prop: auto form to LWM}
The $(\mathfrak{g}, K_\infty^{+})$-module $(V_{f})_{K_\infty^{+}}$ 
is a lowest weight module of weight $(\wedge^{\ell}, -k)$. 
Conversely, if the $K_{\infty}^{+}$-finite part of 
$L^2(\Gamma\backslash {\SO}^+(L\otimes {\R}))$ 
contains a lowest weight module $M$ 
of weight $(\wedge^{\ell}, -k)$, 
then $M=(V_{f})_{K_\infty^{+}}$ for some square-integrable $\Gamma$-modular form $f$ 
of weight $(\wedge^{\ell}, k)$. 
\end{proposition}

\begin{proof}
The first assertion is essentially verified in the proof of Proposition 11.4 in \cite{Ma2}. 
The point is that the $K_\infty^{+}$-representation 
generated by $\alpha \circ \tilde{f}$ 
is isomorphic to $V_{\ell, k}^{\vee}$ by the property (2) in Lemma \ref{lem: lift to G}, 
and the property (3) assures that $(V_{f})_{K_\infty^{+}}$ is a lowest weight module.  

For the second assertion, we take a lowest weight vector $v_{0}$ of $M$. 
Thus $\mathfrak{p}_{-}\cdot v_{0}=0$ and 
$v_{0}$ is a highest weight vector for the $K_\infty^{+}$-action with weight $(\wedge^{\ell}, k)^{\vee}$. 
This vector $v_{0}$ can be extended to a 
$K_\infty^{+}$-homomorphism 
\begin{equation*}
\Phi \colon W^{\vee}\to 
L^2(\Gamma \backslash {\SO}^+(L\otimes \R)) 
\end{equation*}
by sending a highest weight vector of 
$W^{\vee}$ to $v_{0}$. 
Let $\phi$ be the $W$-valued function on 
$\Gamma \backslash {\SO}^+(L\otimes \R)$ 
corresponding to $\Phi$. 
Then $\mathfrak{p}_{-}\cdot v_0 = 0$ implies  $\mathfrak{p}_{-}\cdot \phi = 0$. 
The $K_\infty^{+}$-equivariance of $\Phi$ implies 
the $K_{\infty}^{+}$-invariance of $\phi$, namely 
$\rho(h)(\phi(g \cdot h)) = \phi(g)$ for $h\in K_\infty^{+}$. 
Thus $\phi$ satisfies the properties in Lemma \ref{lem: lift to G}, 
and hence $\phi=\tilde{f}$ for some 
square-integrable modular form $f$. 
Since the above procedure is the inverse of $\tilde{f}\mapsto (V_{f})_{K_\infty^{+}}$, 
we see that $M=(V_f)_{K_\infty^{+}}$. 
\end{proof}

Proposition \ref{prop: hol form via auto rep} now follows from 
Proposition \ref{prop: MF to hol form} and 
Proposition \ref{prop: auto form to LWM}. 

\begin{remark}\label{rmk: irr LWM}
    In fact, 
    $(V_f)_{K_{\infty}^{+}}$ is irreducible even when $l+k\leq n$. 
    This follows from the finiteness of length of highest weight modules 
    (\cite{Hu} \S 1.11) and the semisimplicity of 
    $(V_f)_{K_{\infty}^{+}}$, 
    being a submodule of $L^2(\Gamma\backslash {\SO}^+(L\otimes {\R}))$. 
    In particular, $(V_f)_{K_{\infty}^{+}}$ is contained in the discrete part of 
    the $L^2$-space. 
    This information will be used only in Corollary \ref{cor: exact isom}. 
\end{remark}

\section{Adelization}\label{sec: adelize}

In this section, referring to the following subsections for the terminology, 
we prove the following. 

\begin{proposition}\label{prop: auto rep to MF}
    Let $L$ be a lattice of signature $(2, n)$ with $n\geq 3$ 
    and of class number $1$ with 
    ${\SO}^{+}(L)\ne {\SO}(L)\ne {\rm O}(L)$. 
    Let $n/2<k<n$. 
    Suppose that we have a discrete automorphic representation 
    $\pi=\otimes_{p}\pi_{p}$ of ${\SO}(L\otimes{\A})$ with the following conditions: 

    (1) $\pi_{\infty}\simeq L(\wedge^{n-k}, -k)$. 

    (2) $(\pi_{p})^{K_p}\ne 0$ for every prime 
    $p<\infty$ where 
    $K_p={\SO}(L\otimes {\Zp})$. 

\noindent
Then a smooth projective model of ${\SO}^+(L)\backslash \mathcal{D}$ has a nonzero 
holomorphic $k$-form. 
\end{proposition}
 
In \S \ref{sec: A-packet}, we will construct an automorphic representation 
satisfying these conditions for some special lattices 
(specified in the next \S \ref{sec: lattice}). 

\subsection{Adeles}

Let $L$ be a lattice of signature $(2, n)$ 
as in \S \ref{sec: auto rep to hol form}. 
We denote by $K_{\infty}$ 
the maximal compact subgroup of $\SO(L\otimes \R)$ 
that contains $K_{\infty}^{+}\subset \SO^+(L\otimes \R)$ 
(with index $2$). 
If we write $K_{\infty}=K_{\infty}^{+}\sqcup K_{\infty}^{-}$, 
then $K_{\infty}^{-}$ exchanges $\frakp_{+}$ and $\frakp_{-}$. 
For a prime $p<\infty$ we put 
$K_p :=  {\SO}(L\otimes {\Zp})$. 
This is an open compact subgroup of ${\SO}(L\otimes{\Qp})$. 
For almost all $p$, $K_p$ is furthermore maximal compact. 
We put $K_{{\rm fin}}:=\prod_{p<\infty}K_p$ and 
$K:= K_{{\rm fin}}\times K_{\infty}$. 
The adelic points of ${\SO}(L)$ are defined as the restricted direct product of 
the groups ${\SO}(L\otimes {\Qp})$ with respect to the subgroups $K_p$: 
$$
{\SO}(L\otimes {\Af}) := \prod_{p<\infty} {\SO}(L\otimes {\Qp}), 
$$
$$ 
{\SO}(L\otimes {\A}) := {\SO}(L\otimes {\Af})\times {\SO}(L\otimes {\R}). 
$$
Via the diagonal embedding, 
${\SO}(L\otimes {\Q})$ is a discrete subgroup of ${\SO}(L\otimes {\A})$. 
Its image in ${\SO}(L\otimes {\Af})$ is dense. 

We say that the lattice $L$ has class number $1$ if 
any lattice $L'$ of the same signature as $L$ and satisfying 
$L'\otimes {\Zp} \simeq L\otimes {\Zp}$ for every $p<\infty$ 
is isometric to $L$. 
Since $L$ is indefinite, this condition is often satisfied. 

\begin{lemma}\label{lem: class number 1}
    If $L$ has class number $1$ and ${\SO}(L)\ne {\rm O}(L)$, 
    then $${\SO}(L\otimes {\Af})={\SO}(L\otimes {\Q}) \cdot K_{{\rm fin}}.$$ 
\end{lemma}

\begin{proof}
    Let $g=(g_p)_p$ be an element of ${\SO}(L\otimes {\Af})$. 
    Then $$L':=L\otimes {\Q} \cap g \bigl( L\otimes \prod_{p}{\Zp} \bigr)$$ 
    is a lattice on $L\otimes {\Q}$ (see \cite{Ge} Theorem 9.4). 
    By construction, the ${\Zp}$-lattice 
    $L'\otimes{\Zp}=g_p(L\otimes{\Zp})$ is isometric to $L\otimes {\Zp}$ 
    for each $p<\infty$. 
    By the assumption of class number $1$, 
    there exists an isometry $\gamma$ of $L\otimes{\Q}$ 
    such that $\gamma L' = L$. 
    By the assumption ${\SO}(L)\ne {\rm O}(L)$, 
    we may choose $\gamma$ from ${\SO}(L\otimes{\Q})$. 
    Then the element $\gamma \circ g$ of ${\SO}(L\otimes {\Af})$ preserves 
    $L\otimes (\prod_p{\Zp})$ and hence is contained in $K_{{\rm fin}}$. 
    This implies that $g$ is contained in ${\SO}(L\otimes {\Q}) \cdot K_{{\rm fin}}$. 
\end{proof}

\begin{corollary}\label{cor: class number 1}
    Suppose that $L$ has class number $1$ and that 
    ${\SO}^+(L)\ne {\SO}(L)\ne {\rm O}(L)$. 
    Then the natural embedding 
    ${\SO}(L\otimes {\R})\hookrightarrow {\SO}(L\otimes {\A})$ induces a homeomorphism 
    $$
    {\SO}^+(L) \backslash {\SO}^{+}(L\otimes {\R}) \simeq 
    {\SO}(L\otimes {\Q}) \backslash {\SO}(L\otimes {\A}) / K_{{\rm fin}}. 
    $$
\end{corollary}

\begin{proof}
By substituting Lemma \ref{lem: class number 1} into \cite{GH} Equation (2.23), 
we see that the embedding ${\SO}(L\otimes {\R})\hookrightarrow {\SO}(L\otimes {\A})$ induces 
$$
({\SO}(L\otimes {\Q})\cap K_{{\rm fin}}) \backslash {\SO}(L\otimes {\R}) \simeq 
{\SO}(L\otimes {\Q}) \backslash {\SO}(L\otimes {\A}) / K_{{\rm fin}}. 
$$
We have ${\SO}(L\otimes {\Q})\cap K_{{\rm fin}}={\SO}(L)$ and 
$${\SO}(L)\backslash {\SO}(L\otimes {\R}) = {\SO}^+(L)\backslash {\SO}^{+}(L\otimes {\R})$$ 
by our assumption ${\SO}^+(L)\ne {\SO}(L)$. 
\end{proof}

\subsection{Automorphic representations}

Let $\pi$ be a discrete automorphic representation of 
${\SO}(L\otimes {\A})$ in the sense of \cite{GH} \S 3.7 and \S 6.6. 
This means that $\pi$ is an irreducible 
unitary representation of ${\SO}(L\otimes {\A})$ 
which is equivalent to a sub representation of 
$L^2({\SO}(L\otimes {\Q})\backslash {\SO}(L\otimes {\A}))$. 
Passing to the $K$-finite part, 
this decomposes into a restricted tensor product 
$\pi=\otimes_{p}\pi_{p}$, 
where $\pi_{p}$ for $p<\infty$ is an irreducible admissible representation of ${\SO}(L\otimes {\Qp})$ 
and $\pi_{\infty}$ is an irreducible $(\mathfrak{g}, K_{\infty})$-module. 

In general, the restriction of $\pi_{\infty}$ to 
the subgroup $\SO^{+}(L\otimes \R)$ of $\SO(L\otimes \R)$ 
either remains irreducible or 
decomposes into two irreducible components 
that are exchanged by the action of $K_{\infty}^{-}$. 
When this restriction contains $L(\wedge^{\ell}, -k)$ as 
an irreducible component, it is reducible, 
and the other component is the highest weight module of weight 
$(\wedge^{\ell}, k)$. 
In this case, by abusing notation, we shall simply write 
$\pi_{\infty} \simeq L(\wedge^{\ell}, -k)$. 
    
With these preliminaries, 
Proposition \ref{prop: auto rep to MF} can be deduced as follows. 

\begin{proof}[(Proof of Proposition \ref{prop: auto rep to MF})]
We realize the given representation $\pi$ 
as a sub representation of 
$L^2({\SO}(L\otimes {\Q})\backslash {\SO}(L\otimes {\A}))$. 
By our assumption $(\pi_{p})^{K_p}\ne 0$, 
we may choose a nonzero $K_p$-invariant vector 
$v_p \in \pi_{p}$ for each $p<\infty$. 
(For almost all $p$, we have $\dim (\pi_{p})^{K_p} =1$ 
and so this vector is unique up to constant.) 
Taking the tensor product with the vector $\otimes_{p<\infty} v_p$ 
defines an embedding 
$$
\pi_{\infty} \hookrightarrow \pi \subset L^2({\SO}(L\otimes {\Q})\backslash {\SO}(L\otimes {\A})), 
$$
whose image is $K_{{\rm fin}}$-invariant. 
Restricting this space to the factor ${\SO}(L\otimes {\R})$ of ${\SO}(L\otimes {\A})$, 
we obtain a $(\mathfrak{g}, K_{\infty})$-module $\pi_{\infty}'$ 
isomorphic to $\pi_{\infty}\simeq L(\wedge^{n-k}, -k)$. 
By Corollary \ref{cor: class number 1}, 
$\pi_{\infty}'$ is contained in $L^2({\SO}^+(L)\backslash {\SO}^{+}(L\otimes {\R}))$.  
Then our assertion follows from Proposition \ref{prop: hol form via auto rep}. 
\end{proof}

\section{The lattices}\label{sec: lattice}

In this section, we specialize our lattices 
by a requirement from local representation theory. 
In what follows, the symbol $U$ stands for 
the \textit{integral} hyperbolic plane 
over any given ring of characteristic $0$ ($\Z$ or ${\Zp}$), namely 
the symmetric bilinear form expressed by the Gram matrix 
$\begin{pmatrix} 0 & 1 \\ 1 & 0 \end{pmatrix}$. 

Let $L$ be a lattice of signature $(2, n)$ with $n$ odd. 
Let $p<\infty$ be a prime. 
Following \cite{GH} Definition 2.4.4, we say that 
${\SO}(L\otimes {\Zp})$ is a \textit{hyperspecial subgroup} of 
${\SO}(L\otimes {\Qp})$ if ${\SO}(L\otimes {\Zp})$, 
as an algebraic group over ${\Zp}$, 
has reductive fiber over $\mathbb{F}_p$. 
In general, by \cite{GH} Theorem 2.4.3, 
${\SO}(L\otimes {\Qp})$ has a hyperspecial subgroup 
(not necessarily ${\SO}(L\otimes {\Zp})$) if and only if 
${\SO}(L\otimes {\Qp})$ is split 
($\Leftrightarrow$ quasi-split since $n$ is odd). 
In view of Proposition \ref{prop: everywhere split}, 
we shall assume $n\equiv 1, 3 \bmod 8$ in what follows. 

We consider the following special even lattices: 
\begin{equation}\label{eqn: lattice}
L = 
\begin{cases}
2U \oplus m E_8 \oplus A_1 & n\equiv 3 \bmod 8, \\ 
2U \oplus m E_8 \oplus E_7 & n\equiv 1 \bmod 8. 
\end{cases}
\end{equation}
Note that the second lattice can also be written as 
$\langle 2 \rangle \oplus U \oplus (m+1)E_8$. 
These lattices are the even lattices of minimal discriminant in each $n$. 
(They also appear in \cite{CL} as "$\mathtt{q}$-$\mathtt{i}$-modules".) 
Clearly these lattices satisfy the condition 
${\SO}^{+}(L)\ne {\SO}(L)\ne {\rm O}(L)$ in Proposition \ref{prop: auto rep to MF}. 
It is also classical that they have class number $1$.

\begin{lemma}\label{prop: lattice}
Let $L$ be as in \eqref{eqn: lattice}. 
Then ${\SO}(L\otimes {\Zp})$ is a hyperspecial subgroup of 
${\SO}(L\otimes {\Qp})$ for every prime $p<\infty$. 
\end{lemma}

\begin{proof}
This should be well-known, but since the case $p=2$ is somewhat subtle, 
we provide a proof in this case for the sake of completeness. 
(The case $p>2$ is similar and simpler.) 
It is convenient to pass from symmetric forms to quadratic forms 
already at the level of ${\Z}_{2}$. 
First we recall that $E_8\otimes {\Z}_2 \simeq 4U$, 
as both are even unimodular and have the same discriminant 
(cf.~\cite{Ge} Corollary 8.10). 
Hence we have 
\begin{equation}\label{eqn: LZp}
L\otimes{\Z}_2 \simeq U \oplus \cdots \oplus U \oplus \langle \pm 2 \rangle. 
\end{equation}
Thus the ${\Z}_2$-lattice $L\otimes {\Z}_2$ 
    is the polar symmetric form associated 
    (in the sense of \cite{Milne} \S 24.h) to the quadratic form 
    $$
    q(x_0, \cdots, x_{n+1}) \; = \; 
    \pm x_0^2 + \sum_{i=1}^{(n+1)/2} x_i x_{n+2-i}
    $$
    on ${\Z}_2^{\oplus n+2}$. 
    Since ${\Z}_2$ has characteristic $0$, we can identify 
    $$
    {\SO}(L\otimes {\Z}_2) \simeq {\SO}({\Z}_2^{\oplus n+2}, \: q) 
    $$
    as algebraic groups over $\Z_2$. 
    Then the special fiber of ${\SO}({\Z}_2^{\oplus n+2}, \: q)$ is 
    ${\SO}(\mathbb{F}_2^{\oplus n+2}, \: q)$, 
    and this is semisimple by \cite{Milne} \S 21.j, Example $(B_n)$. 
\end{proof}

\begin{remark}
    We can also give a proof based on the Bruhat-Tits theory as follows. 
    By \eqref{eqn: LZp}, 
    $\SO(L \otimes \bbQ_p)$ is split. 
    Hence, by \cite{Kaletha-Prasad_Bruhat-Tits} Proposition 7.11.7, 
    the special vertex in the Bruhat-Tits building for 
    $\SO(L \otimes \bbQ_p)$ is hyperspecial.
    A direct calculation shows that 
    the subgroup of $\SO(L \otimes \bbZ_p)$ consisting of elements with spinor norm $1$ contains representatives of 
    all elements of the Weyl group of $\SO(L \otimes \bbQ_p)$.
    Therefore $\SO(L \otimes \bbZ_p)$ defines a special vertex 
    by \cite{Kaletha-Prasad_Bruhat-Tits} Proposition 9.9.1,  
    and hence hyperspecial. 
\end{remark}

\section{Arthur multiplicity formula}\label{sec: AMT}

In this section, we recall the theory of 
endoscopic classification of representations of odd orthogonal groups 
following \cite{Arthur} and \cite{AGI+}. 
We take $L$ to be the lattice \eqref{eqn: lattice} of signature $(2, n)$, 
and fix $K_p={\SO}(L\otimes {\Zp})$ as 
the reference hyperspecial subgroup of 
${\SO}(L\otimes {\Qp})$ for every $p<\infty$. 
The choice of a Whittaker data is subsumed.

\subsection{Local A-parameters}\label{ssec: local A}

Let $F$ be either ${\R}$ or ${\Qp}$ with $p< \infty$.
Let $W_{F}$ be the Weil group of $F$. 
The Weil-Deligne group of $F$ is defined by
\[
W_F' = 
\begin{cases}
    W_F \times \SL_2(\bbC) &\text{if $F=\bbQ_p$;}\\
    W_F & \text{if $F= \bbR$.}
\end{cases}
\]
Recall that the Langlands dual group of ${\SO}(L\otimes F)$ is 
$\Sp_{n+1}(\bbC)\times W_{F}$. 
We consider continuous homomorphisms 
$$\psi : W_F' \times \SL_2(\bbC) \rightarrow \Sp_{n+1}(\bbC)$$
such that 
$\psi|_{\SL_2(\bbC)}$ is algebraic and 
$\psi|_{W_F'}$ corresponds to an $L$-parameter with bounded image.
Such a homomorphism $\psi$, considered up to $\Sp_{n+1}({\C)}$-conjugacy, 
is called a \textit{local $A$-parameter}. 
When $F={\Qp}$ with $p<\infty$, 
a local $A$-parameter is said to be \textit{unramified} 
if it is trivial on the inertia subgroup of $W_{F}$.

To a local $A$-parameter $\psi$ we can associate 
a finite set $\Pi(\psi)$ of unitary representations of ${\SO}(L\otimes F)$, 
a finite 2-elementary abelian group $\calS_\psi$, 
and an embedding (determined by our choice of the Whittaker data) 
$$\iota : \Pi(\psi) \hookrightarrow \Irr(\calS_{\psi})$$
into the group of characters of $\calS_{\psi}$. 
The set $\Pi(\psi)$ is called the \textit{$A$-packet} of $\psi$, 
and $\calS_{\psi}$ is called the \textit{component group} of $\psi$. 
If $\eta$ is a character of $\calS_{\psi}$ 
contained in the image of $\iota$, 
we denote by $\pi(\psi,\eta)$ 
the corresponding representation of ${\SO}(L\otimes F)$. 
When $\eta$ is not contained in the image of $\iota$, 
we set $\pi(\psi, \eta) = 0$ for completeness.

In this paper we will take $\eta = \mathbf{1}$ for 
$F={\Qp}$ with $p<\infty$. 
An automorphic representation $\pi$ of ${\SO}(L\otimes {\Qp})$ 
is called \textit{unramified} if $\pi^{K_p}\ne 0$ 
for our hyperspecial subgroup 
$K_p={\SO}(L\otimes {\Zp})$. 

\begin{lemma}[\cite{Arthur}]\label{lemma_A_packet_unramified}
    Let $F={\Qp}$ with $p<\infty$.
    If the local $A$-parameter $\psi$ is unramified, 
    then the trivial character $\mathbf{1}$ is contained in 
    the image of $\iota$ and 
    the representation $\pi(\psi, \mathbf{1})$ is unramified.
\end{lemma}

\begin{proof}
This is essentially contained in Theorem 1.5.1 (a) of \cite{Arthur}.
To supply an argument, let $\phi_\psi$ be the $L$-parameter associated with $\psi$.
By Proposition 7.4.1 of \cite{Arthur}, if $\psi$ is unramified, then $\Pi(\phi_\psi)$ contains an unramified representation.
Hence, there exists $\eta$ such that $\pi(\psi, \eta)$ is unramified.
By Theorem 1.5.1 (a) of \cite{Arthur}, the character $\eta$ is trivial.
Thus $\mathbf{1}$ is contained in the image of $\iota$ and 
$\pi(\psi, \mathbf{1})$ is unramified.
\end{proof}

\begin{remark}\label{remark: unramified local}
    Note that the following form of converse also holds:
    if $\pi(\psi, \eta)$ is unramified, then $\psi$ is unramified and $\eta = \mathbf{1}$.
    Indeed, by \cite{Arthur} Theorem 1.5.1 (a), 
    we have $\eta = \mathbf{1}$.
    Applying \cite{Arthur} Proposition 7.4.1, we see that 
    $\pi(\psi, \mathbf{1})$ is contained in $\Pi(\phi_\psi)$. 
    Hence $\phi_\psi$ is unramified, 
    and so $\psi$ is unramified.
\end{remark}

\subsection{Arthur multiplicity formula}\label{subsec_def_global_A_parameter}

In this subsection, we recall the Arthur multiplicity formula. 

\begin{definition}\label{def: A-packet}
By a \textit{global $A$-parameter}, we mean a formal sum
\[
\psi = \bigoplus_{i=1}^{r} (\pi_i, d_i) 
\]
such that
\begin{itemize}
    \item $\pi_i$ is a unitary cuspidal automorphic representation of $\GL_{m_i}(\bbA)$;
    \item $d_i$ is a positive integer such that $\sum_{i=1}^r m_id_i = n+1$;
    \item If $d_i$ is odd, then $\pi_i$ is symplectic, i.e, $L(s,\pi_i,\Sym^2)$ has a pole at $s=1$;
    \item If $d_i$ is even, then $\pi_i$ is orthogonal, i.e, $L(s,\pi_i,\wedge^2)$ has a pole at $s=1$;
    \item If $i \neq j$ and $\pi_i \isom \pi_j$, then $d_i \neq d_j$.
\end{itemize}
\end{definition}

The \textit{component group} of the global $A$-parameter $\psi$ 
is defined formally by
\[
\calS_\psi = \bigoplus_{i=1}^r (\bbZ/2\bbZ) e_{i}
\]
where the basis element $e_{i}$ corresponds to $(\pi_i, d_i)$. 
We define a character $\vep_\psi$ of $\calS_\psi$ as follows. 
When $r>1$, we set 
\begin{align}\label{def_vep_psi}
\vep_\psi(e_{i}) = \prod_{j \neq i} \vep(\pi_i \times \pi_j)^{\min (d_i, d_j)},
\end{align}
where $\vep(\pi_i \times \pi_j) \in \{\pm1\}$ is 
the root number of the Rankin-Selberg $L$-function $L(s, \pi_i \times \pi_j)$.
By \cite{Arthur} Theorem 1.5.3, 
we have $\vep(\pi_i \times \pi_j)=1$ if $d_i \equiv d_j \bmod 2$. 
When $r=1$, we set $\vep_\psi=\mathbf{1}$. 

\begin{remark}
These are not quite the same as the definition given in \cite{Arthur}, 
but see \cite{CL} \S 8.3.5 for these descriptions.
Our group $\calS_\psi$ is $C_{\psi}=\{ \pm 1 \}^{k}$ 
in the notation of \cite{CL} \S 8.3.5. 
We refer to \cite{CL} p.203 for a procedure 
for calculating $\vep(\pi_i \times \pi_j)$ 
for $A$-parameters appearing in this paper. 
\end{remark}

We can construct an irreducible representation of 
$\SO(L \otimes \bbA)$ from a global 
$A$-parameter $\psi$ by the following process: 
\begin{eqnarray}\label{eqn: A-packet to rep}
     \psi = \bigoplus_{i} (\pi_i, d_i) 
    & \mapsto &(\psi_{p})_{p\leq \infty} := 
    \bigl( \bigoplus_{i} \phi_{i,p} \boxtimes {\Sym}^{d_i-1} \bigr)_{p\leq \infty} \\
    & \mapsto & \pi(\psi, \eta) := 
    \bigotimes_{p\leq \infty} \pi(\psi_p, \eta_p). \nonumber
\end{eqnarray}
 
In the first line, we decompose $\pi_{i}$ into the restricted tensor product 
$\pi_i = \otimes_{p} \pi_{i,p}$ and define 
$\phi_{i,p}$ to be the representation of $W_{F}'$ corresponding to $\pi_{i,p}$ by the Local Langlands Correspondence. 
(The representation $\phi_{i,p}$ is called 
the \textit{$L$-parameter} of $\pi_{i,p}$.) 
The symbol ${\Sym}^{d-1}$ stands for the $(d-1)$-th symmetric tensor 
of the standard representation of $\SL_2(\C)$; 
it has dimension $d$. 
In this way we obtain the local $A$-parameter $\psi_{p}$ for each $p\leq \infty$. 
In the second line, we choose a character $\eta=\prod_{p}\eta_{p}$ 
of $\oplus_{p}S_{\psi_{p}}$ 
such that $\eta_{p}=\mathbf{1}$ for almost all $p$. 
Then we attach the local representation $\pi(\psi_p, \eta_p)$ for each $p$ 
as explained in \S \ref{ssec: local A}. 
Finally, we take their restricted tensor product, which is well-defined by Lemma \ref{lemma_A_packet_unramified}. 

\begin{remark}\label{remark: unramified}
If $\pi_{i}$ is unramified for every $i$, 
then $\phi_{i,p}$ is unramified for every $i$ and $p<\infty$ by a standard property of the Local Langlands Correspondence. 
Hence the local $A$-parameter $\psi_{p}$ is unramified for every $p<\infty$. 
\end{remark}

The process $\psi \mapsto \psi_p$ is called the \textit{localization} of $\psi$. 
It induces a homomorphism 
$\Delta_p \colon \calS_\psi \rightarrow \calS_{\psi_p}$ 
between the component groups. 
We write
$\Delta =(\Delta_{p})_{p}$. 

Arthur's multiplicity formula describes 
the discrete part of 
the space of square-integrable automorphic forms on 
${\SO}(L\otimes \A)$ 
in terms of global $A$-parameters. 
Originally it was proved by Arthur \cite{Arthur} 
in the case of signature $(N, N+1)$; 
the extension to the case of arbitrary signature was done by 
Ta\"{i}bi \cite{Taibi_AMF} and Ishimoto \cite{2024_Ishimoto_Arthur}.
Since the full statement is rather long to state, 
we extract only a part of it that we need later.
\begin{theorem}
[\cite{Arthur}, \cite{Taibi_AMF}, \cite{2024_Ishimoto_Arthur}]\label{theorem_AMF}
    We have the inclusion 
    \[
    L^{2}_{disc}(\SO(L\otimes{\Q}) \backslash \SO(L\otimes{\A})) 
    \; \supset \; \bigoplus_{\psi}\bigoplus_{\eta} 
    \pi(\psi, \eta), 
    \]
    where $\psi$ runs over global $A$-parameters whose archimedean component $\psi_\infty$ is Adams-Johnson 
    (in the sense explained in \S \ref{sec: AJ}), 
    and $\eta=\prod_{p}\eta_{p}$ runs over characters of 
    $\oplus_{p}S_{\psi_{p}}$ such that $\eta_{p}=\mathbf{1}$ for almost all $p$ and satisfying 
    $$\eta \circ \Delta = \vep_\psi.$$
\end{theorem}

\begin{proof}
    See \cite{2024_Ishimoto_Arthur} Theorem 7.2.
    We note that global $A$-parameters with $\psi_\infty$ Adams-Johnson satisfy Hypothesis 7.1 there. 
\end{proof}

\begin{remark}
  Theorem \ref{theorem_AMF} is conditional on the weighted twisted fundamental lemma. 
    See \cite{AGI+} \S 0.4 for more details.
\end{remark}

\section{Adams-Johnson parameters}\label{sec: AJ}

In this section, we give a more detailed account of 
a special class of real $A$-parameters and real $A$-packets, 
and connect it with the lowest weight modules considered in 
\S \ref{sec: auto rep to hol form}. 
We keep the setting of \S \ref{sec: AMT}.

\subsection{Adams-Johnson parameters}\label{subsection_AJ_parameter}

Recall that the real Weil group is 
$W_\bbR = \bbC^\times \sqcup j \, \bbC^\times$ as a set with the group law
$j^2 = -1$ and $jzj^{-1} = \overline{z}$ for $z \in \bbC^\times$. 
We denote by $\sgn\colon W_{\R}\to \{ \pm 1 \}$ 
the sign character of $W_{\R}$, 
namely it is trivial on ${\C}^{\times}$ 
and sends $j$ to $-1$. 
For each positive integer $k$, 
we define a $2$-dimensional representation $\rho_k$ of $W_\bbR$ 
as the following induced representation from $\C^{\times}$: 
\[
\rho_k(z) = 
\begin{pmatrix}
    (z/|z|)^{k} & 0\\
    0& (z/|z|)^{-k}
\end{pmatrix}
,\qquad
\rho_k(j)
=
\begin{pmatrix}
    0&(-1)^k\\
    1&0
\end{pmatrix}.
\]

\begin{definition}\label{def: AJ} 
An $A$-parameter of $\SO(L \otimes \bbR)$ of the form  
        \begin{align}\label{def_Adams_Johnson}
        \psi = 
        \left( \bigoplus_{i=1}^r \rho_{k_i} \boxtimes \Sym^{d_i-1} \right) 
        \: \oplus \: \sgn^\delta \boxtimes \Sym^{d_0-1} 
        \end{align}
is called an \textit{Adams-Johnson parameter} if it satisfies 
the following conditions: 
        \begin{itemize}
            \item $k_i > 0, d_i>0$ for $1 \leq i \leq r$; 
            \item $d_0\geq 0$ (when $d_0=0$, 
            $\sgn^\delta \boxtimes \Sym^{-1}$ means $0$); 
            \item $k_i \equiv d_i \bmod 2$ for any $1 \leq i \leq r$
            and $d_0 \equiv 0 \bmod 2$;
            \item $d_0 + 2 \sum_{i=1}^r d_i = n+1$;
            \item $\delta \in \{0,1\}$;
            \item $k_i-k_{i+1} \geq d_i+d_{i+1}$ for any $1 \leq i < r$ and $k_r \geq d_r+d_0$.
        \end{itemize}
\end{definition}

Here we follow the explicit description in 
\cite{Atobe_A_packets} \S 9.2 and \cite{CR} p.90. 
Note that $k_1>k_2> \cdots$ by the last condition. 
The \textit{infinitesimal character} $\chi_\psi$ of 
the Adams-Johnson parameter $\psi$ is defined as 
the element 
\begin{align}\label{eqn: inf chara}
&\left(\frac{k_1+d_1-1}{2},\frac{k_1+d_1-3}{2},\ldots, \frac{k_1-d_1+1}{2},\right. \\
&\left.\qquad\ldots,\frac{k_r+d_r-1}{2}, \ldots, \frac{k_r-d_r+1}{2},\frac{d_0-1}{2}, \frac{d_0-3}{2},\ldots, \frac{1}{2}\right), 
\notag
\end{align}
of $\bbC^{(n+1)/2}/W$, 
where $W \isom \frakS_{(n+1)/2} \rtimes (\bbZ/2\bbZ)^{(n+1)/2}$ 
is the Weyl group of $\SO_{n+2}(\C)$. 
By the construction \cite{AJ}, 
this is equal to the infinitesimal character 
(in the sense of Harish-Chandra) of every representation 
in the $A$-packet $\Pi(\psi)$. 
The infinitesimal character of the trivial representation of $\SO(L \otimes \bbR)$ is 
the Weyl vector 
\begin{equation}\label{eqn: inf chara trivial}
    \left(
    n/2, n/2-1, \ldots, 1/2    \right).
\end{equation}

Adams and Johnson \cite[Theorem 2.12]{AJ} constructed 
a set $\Pi^{\mathrm{AJ}}(\psi)$ of unitary representations of $\SO(L \otimes \bbR)$ 
with an embedding 
$\iota_\infty^{\mathrm{AJ}} \colon \Pi^{\mathrm{AJ}}(\psi) \hookrightarrow \Irr(\calS_{\psi})$ 
by using cohomological induction. 
Here the component group $\calS_\psi$ is defined formally by
\[
\calS_\psi = \bigoplus_{i=0}^r(\bbZ / 2\bbZ) e_{i,\infty}
\]
with the basis element $e_{i,\infty}$ 
corresponding to $\rho_{k_i} \boxtimes \Sym^{d_i-1}$ if $i \geq 1$, 
and to $\sgn^\delta \boxtimes \Sym^{d_0-1}$ if $i=0$. 
When $d_0=0$, we understand that the indices $i$ start from $i=1$. 
The element  
$z_\psi = \sum_{i} e_{i,\infty}$ 
of $\calS_\psi$ is called the \textit{central element}.

\begin{theorem}[{\cite{AMR}, \cite{MR}}]
Let $\psi$ be an Adams-Johnson parameter. 
     The Adams-Johnson packet $\Pi^{\mathrm{AJ}}(\psi)$ 
     coincides with the $A$-packet $\Pi(\psi)$, and we have 
     $\iota_\infty^{\mathrm{AJ}} = \iota_{\infty}$.
\end{theorem}

\begin{proof}
This is proved in \cite{AMR} Th\'{e}or\`{e}me 1.1 and \cite{MR} \S 5.3. 
\end{proof}

For the computation in the next \S \ref{ssec: LWM AJ}, 
we recall a description of 
$\Pi^{\mathrm{AJ}}(\psi)$ and $\iota_{\infty}^{\mathrm{AJ}}$ 
following \cite{CR} Appendix A.
Let $T_c$ be the maximal torus of $\SO_{n+2}(\R)$ defined by
    \[
    T_c = \{\diag(1, r_{\theta_1},\ldots,r_{\theta_{(n+1)/2}}) \mid \theta_1, \ldots, \theta_{(n+1)/2} \in \bbR\}, \quad r_{\theta} =
    \begin{pmatrix}
        \cos\theta & \sin\theta\\
        -\sin\theta & \cos\theta
    \end{pmatrix}
    \]
and $T\subset \SO_{n+2}(\C)$ be the complexification of $T_c$. 
We consider the group
\begin{align*}
T[2] 
& := \{t \in T \mid t^2 =1\}\\
& = \{\diag(1,\pm 1_2, \ldots,\pm 1_2)\} \isom (\bbZ/2\bbZ)^{(n+1)/2}.
\end{align*}
For each $t \in T[2]$, we can associate 
a pure inner form $G_t\subset \SO_{n+2}(\C)$ of $\SO_{n+2}(\R)$. 
We have $G_t\simeq \SO(p,q)$ 
where $p/2$ is the number of $-1_2$ in $t$ 
and $p+q = n+2$. 
In particular, our $\SO(L\otimes \R)$ corresponds to 
an element of $T[2]$ with only one $-1_2$ component. 
We choose one such element, say $t_0$. 

Let $\psi$ be an Adams-Johnson parameter. 
It determines a Levi subgroup $L_{\psi}$ of $\SO_{n+2}(\C)$. 
For each $t\in T[2]$, 
$L_t := G_t \cap L_{\psi}$ is a real form of $L_{\psi}$ 
containing $T_c$. 
The Lie algebra of $L_t$ is described as 
    \begin{equation}\label{descrip_l_t}
    \frakl_{t} \simeq 
    \fraku(p_1, q_1) \times \cdots \times \fraku(p_r,q_r) \times \mathfrak{so}(p_0,q_0)
    \end{equation}
where $p_1$ is the number of $-1_{2}$ 
in the first $d_1$ components in $t$, 
$p_2$ is that in the next $d_2$ components and so on, 
and $p_0/2$ is that in the last $d_0/2$ components. 
$q_i$ are determined by 
$p_i+q_i = d_i$ for $i>0$ 
and $p_0+q_0=d_0+1$. 
See, e.g., \cite{Atobe_A_packets} \S 9. 

For each $t\in T[2]$, let 
\begin{equation}\label{eqn: pit}
\pi_t = A_{\mathfrak{q}_{t}}(\lambda-\rho)
\end{equation} 
be the cohomological induction to $G_t$ 
from the character of $L_t$ whose restriction to $T_c$ 
is $\lambda-\rho$, 
where $\lambda$ is the vector \eqref{eqn: inf chara} of $\C^{(n+1)/2}$ 
(not modulo $W$) 
and $\rho$ is the Weyl vector \eqref{eqn: inf chara trivial}. 
Here $\mathfrak{q}_t$ is the (standard) 
$\theta$-stable parabolic subalgebra of 
$\mathfrak{so}_{n+2}(\C)$ with Levi part  
$(\mathfrak{l}_{t})_{\C}$. 
Then $\Pi^{\mathrm{AJ}}(\psi)$ is defined as 
\[
\Pi^{\mathrm{AJ}}(\psi) = 
\{ \pi_{w(t_0)} \mid  w \in W \}, 
\]
where $\pi_{w(t_0)}$ is regarded as a representation of 
$\SO(L\otimes \R)$ via the isomorphism 
$\SO(L\otimes \R)\simeq G_{t_0}\simeq G_{w(t_0)}$.

Let $\widehat{T}$ be the centralizer of 
$\varphi_{\psi}({\C}^{\times})$ in $\Sp_{n+1}(\bbC)$.
We have a natural pairing between $T[2]$ and $\widehat{T}[2]$. 
The component group $\calS_\psi$ is realized as 
a subgroup of $\widehat{T}[2]$ (\cite{CR} Lemma A.4). 
For $t\in T[2]$, we denote by $\tau_0(t)$ 
the corresponding character $\widehat{T}[2]\to \Z/2\Z$. 
Then $\iota_\infty^\mathrm{AJ}$ is given by 
$$
\iota_\infty^\mathrm{AJ}(\pi_t) = (\tau_0(t) - \rho^\vee)|_{\calS_\psi}, 
$$
where $\rho^\vee$ is the Weyl vector of $\Sp_{n+1}(\bbC)$.

Finally, we go back to global $A$-parameters and 
recall the explicit description of 
the localization map $\Delta_\infty$.
Let 
\[
\psi = \bigoplus_{i=1}^r (\pi_i, d_i), \qquad \calS_\psi = \bigoplus_{i=1}^r (\bbZ/2\bbZ) e_{i}
\]
be a global $A$-parameter and its component group.
Suppose that the localization $\psi_\infty$ at $p=\infty$ 
is Adams-Johnson.
We write $\psi_\infty$ and its component group as
\[
\psi_\infty = \bigoplus_{j=1}^t \sigma_j \boxtimes {\rm Sym}^{m_j-1}, \qquad \calS_{\psi_\infty} = \bigoplus_{j=1}^t (\bbZ/2\bbZ)e_{j, \infty}
\]
where $\sigma_j$ are irreducible representations of $W_\bbR$. 
The localization defines a partition 
$\{1,\ldots,t\}=I_1 \sqcup \cdots \sqcup I_r$ such that 
the $L$-parameter for $\pi_{i,\infty}$ is 
$\bigoplus_{j \in I_i} \sigma_j$ and $m_j = d_i$ for $j \in I_i$. 
Thus we have 
\[
\psi_{\infty} \isom \bigoplus_{i=1}^r 
 \bigoplus_{j \in I_i} \sigma_j  
\boxtimes {\rm Sym}^{d_i-1}. 
\]
Then the map $\Delta_\infty$ is given by
\[
\Delta_\infty(e_{i}) = \sum_{j \in I_i} e_{j,\infty}.
\]

\subsection{Special lowest weight modules}\label{ssec: LWM AJ}

Let $n/2< k < n$. 
Recall from \S \ref{sec: auto rep to hol form} and 
\S \ref{sec: adelize} that 
$L(\wedge^{n-k}, -k)$ stands for 
the irreducible lowest weight module of weight  
\[
(\underbrace{1,\ldots,1}_{n-k}, \underbrace{0,\ldots,0}_{k-(n+1)/2}, -k)
\]
for $\SO^{+}(2, n)$ (and also for $\SO(2, n)$).   

\begin{lemma}\label{lem: inf chara}
$L(\wedge^{n-k}, -k)$ is unitarizable and 
its infinitesimal character is given by \eqref{eqn: inf chara trivial}, 
the same as that of the trivial representation.
\end{lemma}

\begin{proof}
    The unitarizability follows from \cite{EHW_83}. 
    The infinitesimal character of a lowest weight module is 
    the sum of the lowest weight and the Weyl vector. 
    In our case, this is calculated as 
    \begin{align*}
    &(\underbrace{1,\ldots,1}_{n-k}, \underbrace{0,\ldots,0}_{k-(n+1)/2}, -k) + (n/2-1,\ldots, 1/2, n/2) \\
    &= (\underbrace{n/2,\ldots,k-n/2+1}_{n-k}, \underbrace{k-n/2-1,\ldots,1/2}_{k-(n+1)/2},n/2-k).
    \end{align*}
    By the Weyl group action, 
    this can be transformed to \eqref{eqn: inf chara trivial}.
\end{proof}

We study $L(\wedge^{n-k}, -k)$ in terms of cohomological induction.
Let $A_{\mathfrak{q}}=A_{\mathfrak{q}}(0)$ 
be the cohomological induction 
to $\SO(2,n)$ induced from the trivial character
of a $\theta$-stable standard parabolic subalgebra 
$\mathfrak{q} = \mathfrak{l}_{\C} \oplus \mathfrak{u}$ of 
$\mathfrak{so}_{n+2}(\bbC)$. 

\begin{lemma}\label{lem_coh_ind}
    Suppose that $\frakl$ is isomorphic to
    \begin{equation}\label{eqn: special l}
    \fraku(1,d_1-1) \oplus \fraku(0,d_2) \oplus \cdots \oplus \fraku(0,d_r)\oplus \mathfrak{so}(0, d_0+1).
    \end{equation}
    Then $A_\frakq \isom L(\wedge^{n-k}, -k)$ 
    where $k=n+1-d_1$. 
\end{lemma}
\begin{proof}
    By Mackey theory, the restriction of $A_{\frakq}$ 
    to $\SO^+(2,n)$ is the direct sum of cohomological inductions 
    induced from $\frakq$ and its opposite.
    Hence, we may assume 
    $\fraku \cap \frakp_\bbC \subset \frakp_+$.
    By the description of $K_\infty^+$-types of $A_\frakq$ 
    in \cite{VZ} Theorem 2.5, $A_\frakq$ is a lowest weight module 
    of weight $2 \rho(\fraku \cap \frakp_\bbC)$.
    Here $2 \rho(\fraku \cap \frakp_\bbC)$ is the sum of roots in 
    $\fraku \cap \frakp_\bbC$.
    The adjoint action of $K_\infty^+$ on $\frakp_+$ and $\frakp_-$ has highest weights $(1,0,\ldots,0,-1)$ and $(1,0,\ldots,0,1)$ 
    (cf.~ \S \ref{sec: auto rep to hol form}).
    Let $e_i\in \bbC^{(n+1)/2}$ be the $i$-th vector 
    in the standard basis. 
    The roots of $\frakp_+$ and $\frakp_-$ are given by
    \begin{equation}\label{eqn: root p+}
    \Delta(\frakp_+) = \{\pm e_i - e_{(n+1)/2} \mid 1 \leq i \leq (n-1)/2\} \cup \{- e_{(n+1)/2}\}, 
    \end{equation}
    \[
    \Delta(\frakp_-) = \{\pm e_i + e_{(n+1)/2} \mid 1 \leq i \leq (n-1)/2\} \cup \{e_{(n+1)/2}\}.
    \]
    By the assumption on $\frakl$, we have 
    \begin{align}\label{comp_u_cap_p}
    \Delta(\fraku\cap \frakp_{\C}) 
    &= \{e_j - e_{(n+1)/2} \mid 1 \leq j \leq (n-1)/2\} \\
    &\qquad \cup 
\{-e_j - e_{(n+1)/2} \mid d_1 \leq j \leq (n-1)/2\} \cup \{-e_{(n+1)/2}\}. \notag
    \end{align}
    It follows that $2\rho(\fraku \cap \frakp_{\C})$ is given by 
    \begin{align*}
    & \sum_{i=1}^{d_1-1}(e_i-e_{(n+1)/2}) + \sum_{j=d_1}^{(n-1)/2}((e_j-e_{(n+1)/2})+ (-e_j-e_{(n+1)/2})) - e_{(n+1)/2}\\
    &= e_1 + \cdots + e_{d_1-1} - (n-d_1+1)e_{(n+1)/2}\\
    &= (\underbrace{1,\ldots,1}_{n-k},0,\ldots,0, -k).
    \end{align*}
    Hence $A_\frakq \isom L(\wedge^{n-k}, -k)$.
\end{proof}

By combining this computation with 
the construction of Adams-Johnson packets, 
we obtain the following.

\begin{proposition}\label{lemma_Adams_Johnson_packets}
Let $n/2<k<n$. 
        Suppose that $\psi$ is an Adams-Johnson parameter of the form (\ref{def_Adams_Johnson}) with $d_1 = n-k+1$ such that the infinitesimal character of $\psi$ is given by \eqref{eqn: inf chara trivial}. 
        Then $\Pi(\psi)$ contains $L(\wedge^{n-k}, -k)$, 
        and the character $\eta$ corresponding to $L(\wedge^{n-k}, -k)$ is given by 
        \begin{equation}\label{eqn: etaAJ}
        \eta(e_{i,\infty})
        =
        \begin{cases}
            (-1)^{[(d_0+2)/4]} & \text{if $i=0$;}\\
            (-1)^{1+d_1(d_1+1)/2+((n+1)/2-d_1)d_1} & \text{if $i=1$;}\\
            (-1)^{d_i(d_i+1)/2 + ((n+1)/2-d_{\leq i})d_i} &\text{if $i>1$.}
        \end{cases}
        \end{equation}
        Here $d_{\leq i} = \sum_{j=1}^i d_j$. 
        Moreover, we have $\eta(z_\psi) = 1$. 
\end{proposition}

\begin{proof}

    Let $t_0 = \diag(1,-1_2, 1,\ldots,1) \in T[2]$. 
    By \eqref{eqn: pit}, our assumption on the infinitesimal character
    implies that $\pi_{t_0}=A_{\mathfrak{q}_{t_0}}$, i.e., 
    induced from the trivial character. 
    By \eqref{descrip_l_t}, our choice of $t_0$ implies that 
    $\mathfrak{l}_{t_0}$ is of the form \eqref{eqn: special l}. 
    Thus we can apply Lemma \ref{lem_coh_ind} 
    to see that $\pi_{t_0}\simeq L(\wedge^{n-k}, -k)$.

    Next we compute the corresponding character. 
    Let us describe $\tau_0(t_0)$ and $\rho^\vee$ 
    in terms of $\widehat{T}$. 
    According to the given expression of $T[2]$, 
    we write 
    \[
    \widehat{T}[2] = \bigoplus_{i=1}^{(n+1)/2} (\Z/2\Z) \theta_{i}. 
    \]
    By our choice of $t_0$, 
    $\tau_0(t_0)$ is given by
    $\theta_1 \mapsto -1$ and $\theta_i\mapsto 1$ for $i>1$. 
    Similarly, since 
    \[
    \rho^\vee = ((n+1)/2, (n-1)/2, \cdots 1), 
    \] 
    the character of $\widehat{T}[2]$ corresponding to $\rho^\vee$ is given by $\theta_i \mapsto (-1)^{(n+3)/2-i}$.
     On the other hand, $\calS_\psi$ is embedded in  
     $\widehat{T}[2]$ by 
    \[
    e_{i,\infty} \mapsto \sum_{j=d_{\leq i-1}+1}^{d_{\leq i}} \theta_j 
    \quad (i>0), \qquad 
    e_{0,\infty} \mapsto \sum_{j=d_{\leq r}+1}^{(n+1)/2} \theta_j. 
    \]
    From this we calculate 
    \[
    \eta(e_{1,\infty}) = (-1)^{1+(n+1)/2+(n-1)/2+ \cdots +(n+3-2d_1)/2} = (-1)^{1 + d_1(d+1)/2+((n+1)/2-d_1)d_1}, 
    \]
    \begin{align*}
    \eta(e_{i,\infty}) 
    &= (-1)^{(n+1 - 2d_{\leq i-1})/2 + (n-1 - 2d_{\leq i-1})/2 + \cdots + (n+3 - 2d_{\leq i})/2} \\
    &= (-1)^{d_i(d_i+1)/2+((n+1)/2-d_{\leq i})d_i}
    \end{align*}
    for $i > 1$, and 
    \begin{equation*}
    \eta(e_{0,\infty}) 
     = (-1)^{d_0/2 + (d_0/2-1) + \cdots +1} 
     = (-1)^{d_0(d_0+2)/8}= (-1)^{[(d_0+2)/4]}
    \end{equation*}
    by $d_0$ even. 
    Finally, we have 
    \[
    \eta(z_{\psi}) = (-1) \cdot (-1)^{(n+1)/2+(n-1)/2 + \cdots + 1} = (-1)^{1+(n+1)(n+3)/8} = 1 
    \]
    by $n\equiv 1, 3 \bmod 8$.
\end{proof}

When the Adams-Johnson parameter $\psi$ does not contain  
a $\textrm{sgn}^{\delta}$ factor, 
the condition that the infinitesimal character of $\psi$ 
is equal to \eqref{eqn: inf chara trivial} amounts to the equalities 
\begin{equation}\label{eqn: AJ equality}
    k_1+d_1= n+1, \quad k_i-k_{i+1}=d_i+d_{i+1} \: \:  (1\leq i < r), 
    \quad k_r=d_r. 
\end{equation}
When there is a $\textrm{sgn}^{\delta}$ factor, 
the last condition is replaced by $k_r=d_r+d_0$:  
\begin{equation}\label{eqn: AJ equality II}
    k_1+d_1= n+1, \quad k_i-k_{i+1}=d_i+d_{i+1} \: \:  (1\leq i < r), 
    \quad k_r=d_r+d_0. 
\end{equation}

For later use in \S \ref{ssec: classify}, 
we also prepare the following. 

\begin{lemma}\label{lem_lowest_degree}
    Let $n \equiv 1 \bmod 8$.  
    If the packet $\Pi(\psi)$ of 
    an Adams-Johnson parameter $\psi$ contains 
    $L(\wedge^{(n-1)/2}, -(n+1)/2)$, then 
    $\psi = \rho_{(n+1)/2} \boxtimes \Sym^{(n-1)/2}$. 
\end{lemma}

\begin{proof}
We write $\psi$ in the form \eqref{def_Adams_Johnson}. 
We shall show that $r=1$, $d_0=0$ in the notation there; 
this implies $d_1=(n+1)/2$, and then 
$k_1=(n+1)/2$ by Lemma \ref{lem: inf chara} and 
\eqref{eqn: AJ equality}. 

By our assumption and Lemma \ref{lem: inf chara}, we have 
$L(\wedge^{[n/2]}, [-n/2])\simeq A_{\mathfrak{q}_t}$ 
for a $\theta$-stable parabolic subalgebra $\mathfrak{q_t}$. 
Let $\mathfrak{u}_t$ be the nilpotent part of $\mathfrak{q}_t$. 
Let $(p_i, q_i)_i$ be as in \eqref{descrip_l_t}. 
Since $\sum_i p_i=1$, we have $p_{i_0}=1$ for some $i_0$ 
and $p_i=0$ for $i\ne i_0$. 

\begin{claim}\label{claim}
   We have $\dim (\fraku_{t} \cap \frakp_{\C})=n+1-d_{i_0}$. 
\end{claim}

\begin{proof}
Similarly to \eqref{comp_u_cap_p}, 
the set of roots of $\fraku \cap \frakp_\bbC$ is calculated as 
\[
\begin{cases}
    \Delta(\frakp_+) \setminus \{e_{j} - e_{(n+1)/2} \mid 
    d_{\leq (i_0-1)} + 1 \leq j \leq d_{i_0}-1\} & \text{if $i_0 > 0$;}\\
    \{\pm e_i - e_{(n+1)/2} \mid 1 \leq i \leq (n+1)/2 - d_0/2\} & \text{if $i_0 =0$,}
\end{cases}
\]
where $d_{\leq 0} = 0$. 
This proves our claim. 
\end{proof}

Recall from \cite{VZ} that 
$\dim (\mathfrak{u}_{t} \cap \frakp_{\C})$ is the minimal degree 
of nonzero Lie algebra cohomology 
$H^{\ast}(\frakg, K_{\infty}^{+}; A_{\mathfrak{q}_{t}})$. 
Combining this with Claim \ref{claim}, 
we see that $d_{i_0}$ is independent of $(\psi, t)$ 
with $L(\wedge^{[n/2]}, [-n/2])\simeq A_{\mathfrak{q}_t}$. 
On the other hand, 
by Proposition \ref{lemma_Adams_Johnson_packets} (and its proof), 
we see that 
$\rho_{(n+1)/2} \boxtimes \Sym^{(n-1)/2}$ and 
$t_0$ is one such pair. 
This implies that $d_{i_0}=(n+1)/2$ for the given $(\psi, t)$. 
By our assumption $n\equiv 1 \bmod 8$, 
$d_{i_0}$ is odd, and so $i_0>0$. 
By the condition $d_0+2\sum_{i=1}^{r} d_i=n+1$ 
in Definition \ref{def_Adams_Johnson}, 
we find that $d_0=0$ and $r=1$. 
\end{proof}

\section{Construction of A-parameters}\label{sec: A-packet}

In this section, we construct global $A$-parameters 
required for the construction of holomorphic differential forms.

\subsection{Criterion}\label{ssec: recipe}

Let $L$ be a lattice as in \eqref{eqn: lattice}.  
Let $\psi=\bigoplus_{i}(\pi_{i}, d_i)$ be a global $A$-parameter of 
${\SO}(L\otimes {\A})$ whose archimedean component 
$\psi_\infty$ is Adams-Johnson. 
We denote by $\psi_{i,\infty}$ the component of $\psi_\infty$ 
associated to $(\pi_{i}, d_{i})$ 
by the procedure \eqref{eqn: A-packet to rep}. 
We normalize the order of the indices $i$ as follows. 
First, if there exists $i$ such that 
$\psi_{i,\infty} = \sgn^\delta \boxtimes \Sym^{d_{i}-1}$, 
we understand this $i$ as $0$. 
In other cases, we understand $i\geq 1$, 
and write $\psi_{i,\infty}$ as 
\[
\psi_{i,\infty} = 
\bigoplus_{j=1}^{r_i} \rho_{k_{i,j}} \boxtimes \Sym^{d_{i}-1} \; 
( {\rm possibly} \: \oplus \sgn^{\delta} \boxtimes \Sym^{d_{i}-1} )
\]
where the indices $j$ are ordered so that
$k_{i,1}>\cdots>k_{i,r_i}$.
Then, the indices $i\geq 1$ are ordered so that $k_{1,1}>k_{2,1}>\cdots$. 
Therefore   
\[
k_{1,1} = \max_{i,j\geq 1} k_{i,j}, 
\]
and this is $k_1$ in \eqref{def_Adams_Johnson}.
In particular, $d_1$ here coincides with 
$d_1$ in \eqref{def_Adams_Johnson}. 

By Proposition \ref{lemma_Adams_Johnson_packets}, 
the local $A$-packet $\Pi(\psi_{\infty})$ of $\psi_{\infty}$ 
contains the irreducible lowest weight module 
$L(\wedge^{n-k}, -k)$ 
if the infinitesimal character of $\psi_{\infty}$ is 
\eqref{eqn: inf chara trivial} and $d_1=n-k+1$. 
In this case, the corresponding character $\eta_{\infty}$ 
of the local component group $\calS_{\psi_{\infty}}$ 
is given by \eqref{eqn: etaAJ}. 
Our recollection in \S \ref{sec: AMT} and computation in 
\S \ref{sec: AJ} were designed to converge to the following criterion. 

\begin{proposition}\label{prop: criterion}
Let $L$ be as in \eqref{eqn: lattice}.   
Let $n/2<k<n$. 
Suppose that we have a global $A$-parameter 
$\psi=\bigoplus_{i}(\pi_{i}, d_i)$ of ${\SO}(L\otimes {\A})$ satisfying the following conditions: 
\begin{enumerate}
    \item The automorphic representation $\pi_{i}$ is unramified for each $i$. 
    \item $\psi_{\infty}$ is an Adams-Johnson parameter
    satisfying \eqref{eqn: AJ equality} or \eqref{eqn: AJ equality II}. 
    \item $d_1=n-k+1$. 
    \item The character $\eta_{\infty}$ of $\calS_{\psi_{\infty}}$ 
    defined by \eqref{eqn: etaAJ} 
    satisfies $\eta_{\infty}\circ \Delta_{\infty}=\varepsilon_{\psi}$. 
\end{enumerate}
Then a smooth projective model of $\mathcal{F}_{n}$ has a nonzero holomorphic $k$-form. 
\end{proposition}

\begin{proof}
We put $\eta_{p}=\mathbf{1}$ for $p<\infty$ and set $\eta=\prod_{p\leq \infty}\eta_{p}$. 
Then the assumption $\eta_{\infty}\circ \Delta_{\infty}=\varepsilon_{\psi}$ 
implies $\eta\circ \Delta=\varepsilon_{\psi}$. 
Hence 
$$
\pi(\psi, \eta) = \bigotimes_{p<\infty}\pi(\psi_{p}, \mathbf{1}) \otimes L(\wedge^{n-k}, -k)
$$
is automorphic by Theorem \ref{theorem_AMF}. 
Since $\psi$ is unramified, $\psi_{p}$ is unramified for every $p<\infty$ 
by Remark \ref{remark: unramified}. 
Hence $\pi(\psi_{p}, \mathbf{1})$ is unramified for every $p<\infty$ by Lemma \ref{lemma_A_packet_unramified}. 
Thus $\pi(\psi, \eta)$ is unramified with 
archimedean component $L(\wedge^{n-k}, -k)$. 
Then we can apply Proposition \ref{prop: auto rep to MF}.  
\end{proof}

\subsection{The input}\label{ssec: input}

In the rest of \S \ref{sec: A-packet}, 
we will construct global $A$-parameters $\psi$ satisfying the conditions in Proposition \ref{prop: criterion}. 
The cuspidal representations $\pi_{i}$ will be one of the following types: 
\begin{itemize}
\item The trivial character $\mathbf{1}$ of $\GL_1(\bbA)$. 
\item A cuspidal representation $\sigma_{k}$ of $\GL_2(\bbA)$ 
generated by an eigenform of weight $k$ for $\SL_2(\Z)$.  
(\S \ref{subsection_GL_2_autom_rep}) 
\item The symmetric square $\Sym^2 \sigma_{k}$ of $\sigma_k$, 
which is a cuspidal representation of $\GL_3(\bbA)$. 
(\S \ref{subsection_GL_2_autom_rep}) 
\item A cuspidal representation $\tau_{j,k}$ of $\GL_4(\bbA)$ 
associated to a Siegel eigenform of weight $\Sym^j \otimes \det^k$ for $\Sp_4(\Z)$. (\S \ref{subsection_autom_rep_GL_4_5})
\item A cuspidal representation $\xi_{j,k}$ of $\GL_5(\bbA)$ 
associated to a Siegel eigenform of weight $\Sym^j \otimes \det^k$ for $\Sp_4(\Z)$. (\S \ref{subsection_autom_rep_GL_4_5})
\end{itemize}
With these notations, our results can be summarized as follows. 
We begin with some ``easy" $A$-parameters, 
by which we construct $k$-forms for relatively small $k$. 

\begin{proposition}\label{prop: k minimum}
The $A$-parameters in Table \ref{table: simple parameter}
exist and satisfy the conditions in Proposition \ref{prop: criterion} 
with $k=n+1-d_1$. 
In the last case, $\ell$ is an arbitrary integer 
with $\ell \equiv 1-n$ mod $8$ 
in the range $0< \ell \leq (n+1)/3$. 
\end{proposition}

Next we construct $k$-forms for odd $k$ more systematically. 

\begin{proposition}\label{prop: input k odd}
The $A$-parameters in Table \ref{table: input odd k} exist 
and satisfy the conditions in Proposition \ref{prop: criterion} 
with $k=n+1-d$. 
Here $d$ is an arbitrary integer satisfying 
$$
\begin{cases}
1 < d \leq (n-21)/2,  \; d\equiv n \bmod 4 &  
\quad \textrm{if} \: \:  k\equiv 1 \bmod 4 \\ 
1< d \leq (n-21)/4, \; d \equiv n+2 \bmod 4 & 
\quad \textrm{if} \: \:  k\equiv 3 \bmod 4. 
\end{cases}
$$
The bound of $n$ is so that the set of $d$ satisfying these conditions 
is non-empty. 
\end{proposition}

The case of even $k$ is more complicated.  

\begin{proposition}\label{prop: input k even}
The $A$-parameters in Table \ref{table: input even k} exist 
and satisfy the conditions in Proposition \ref{prop: criterion} 
with $k=n+1-d$. 
Here $d$ is an arbitrary natural number satisfying 
\begin{equation*}
    \begin{cases}
        1< d \leq  (n+1)/5, \quad d \equiv n+3 \bmod 8, 
        & \quad \textrm{if} \: \:  k\equiv 6 \\ 
        1< d < (n-1)/7, \quad d \equiv n+1 \bmod 8, 
        & \quad \textrm{if} \: \:  k\equiv 0 \\ 
        1< d \leq (n-3)/9, \quad d \equiv 2 \bmod 8, 
        & \quad \textrm{if} \: \: (n, k)\equiv (3, 2) \\ 
        1< d \leq (n-19)/5, \quad d \equiv 0 \bmod 8, 
        & \quad \textrm{if} \: \:  (n, k)\equiv (1, 2) \\ 
        1< d \leq (n-11)/7, \quad d \equiv n-3 \bmod 8, 
        & \quad \textrm{if}  \: \:  k\equiv 4 \\ 
    \end{cases}
\end{equation*}
The bound of $n$ is so that the set of $d$ satisfying these conditions 
is non-empty, except that 
$n\geq 41$ in the case $(n, k)\equiv (1, 0)$. 
\end{proposition}

Finally, we consider the case $n=25, 27, 33, 35$. 
The above systematic constructions already cover 
many degrees $k$ for those $n$. 
We list the remaining $A$-parameters in 
Table \ref{table: input small n}. 

\begin{proposition}\label{prop: input n small}
The $A$-parameters in Table \ref{table: input small n} exist 
and satisfy the conditions in Proposition \ref{prop: criterion}. 
\end{proposition}

\begin{table}
    \centering
    \caption{Easy $A$-parameters}
    \label{table: simple parameter}
    \begin{tabular}{c|c|c}
         $n \bmod 8$ & \text{$A$-parameters} & bound of $n$ \\
         \hline
    $1$ & $(\sigma_{(n+3)/2}, \; (n+1)/2)$ & $n\geq 33$  \\
    \hline 
    $3$ & $(\tau_{(n-3)/4, (n+9)/4}, (n+1)/4)$ & $n\geq 35$ \\
    \hline
    $1, 3$ & $(\sigma_{(n+\ell+3)/2}, \; (n+1-\ell)/2) \oplus 
    (\mathbf{1}, \: \ell)$ & $n\geq 25$    
    \end{tabular}
\end{table}

\renewcommand{\arraystretch}{1.2}
\begin{table}
    \centering
    \caption{Odd $k$}
    \label{table: input odd k}
    \begin{tabular}{c|c|c}
         $n \bmod 8$& $k \bmod 4$ &\text{$A$-parameters} \\
         \hline
    $1$ & $1$ & $(\sigma_{n+2-d}, \, d) \oplus (\sigma_{n+1-2d},\, 1) 
        \oplus (\sigma_{(n+1)/2-d}, \, (n-1)/2-d)$  \\
    \hline
    $3$ & $1$ &  $(\sigma_{n+2-d}, \, d) \oplus 
        (\sigma_{(n+3)/2-d}, \, (n+1)/2-d)$  \\
    \hline
    $1$ & $3$ &  $(\tau_{n-3d, d+2}, \, d) \oplus 
        (\sigma_{(n+3)/2-2d}, \, (n+1)/2-2d)$  \\
    \hline
    $3$ &$3$ &  $(\tau_{n-3d, d+2}, \, d) \oplus 
        (\sigma_{n-4d+1}, 1) \oplus 
        (\sigma_{(n+1)/2-2d}, \, (n-1)/2-2d)$  
    \end{tabular}
\end{table}

\renewcommand{\arraystretch}{1.2}
\begin{table}
    \centering
    \caption{Even $k$}
    \label{table: input even k}
    \begin{tabular}{c|c|c}
         $n \bmod 8$& $k \bmod 8$ &\text{$A$-parameters} \\
         \hline
    $1, 3$ & $6$ &  
    $(\Sym^2\sigma_{(n+3-d)/2}, \, d) \oplus 
    (\sigma_{(n+3-d)/2},  \, (n+1-3d)/2)$   \\
    \hline
    $1, 3$ & $0$ &  
    \begin{tabular}{c}
    $(\xi_{2d, (n-5d+3)/2}, \, d) \oplus 
    (\sigma_{n+1-3d}, \, d+1)$ \\ 
    $\oplus (\sigma_{(n+1-5d)/2}, \, (n-1-7d)/2)$  
    \end{tabular} \\
    \hline
    $3$ & $2$ & 
    \begin{tabular}{c}
    $(\xi_{2d, (n+3-5d)/2}, d) \oplus 
    (\tau_{2d, (n+3-5d)/2}, d+1)$  \\ 
    $\oplus (\sigma_{(n+3-3d)/2}, (n-3-9d)/2)$ 
    \end{tabular} \\
    \hline
    $1$ & $2$ & 
    \begin{tabular}{c}
    $(\Sym^2\sigma_{(n-d+3)/2},d)$  \\ 
    $\oplus (\sigma_{(n+23-d)/2}, (n-19-3d)/2) 
     \oplus (\tau_{d+4,7}, 5)$
   \end{tabular} \\ 
    \hline
    $1$ &$4$ & 
    \begin{tabular}{c}
    $(\xi_{2d, (n-5d+3)/2}, d) \oplus (\sigma_{n-3d+1}, d+1)$ \\ 
    $\oplus \bigoplus_{m=1}^{(n-7d-11)/2} (\sigma_{n-6d+1-2m}, 1) \oplus (\sigma_{d+6}, 5)$ 
    \end{tabular} \\ 
    \hline
    $3$ & $4$ & 
    \begin{tabular}{c}
    $(\xi_{2d, (n-5d+3)/2}, d) \oplus 
    (\sigma_{n-3d+1}, d+1)$ \\ 
    $\oplus \bigoplus_{m=1}^{(n-7d-7)/2} (\sigma_{n-6d+1-2m}, 1) 
    \oplus (\sigma_{d+4}, 3)$ 
    \end{tabular}
    \end{tabular}
\end{table}

\begin{table}
    \centering
    \caption{Remaining $A$-parameters in $n\leq 35$}
    \label{table: input small n}
    \begin{tabular}{c|c|c}
         $n$& $k$ &\text{$A$-parameters} \\
         \hline
    $27$ & $23$ & $(\tau_{8,9}, 5) \oplus (\sigma_{18}, 1) \oplus (\sigma_{16}, 1) \oplus (\mathbf{1}, 4)$  \\
    \hline
    $33$ & $25$ &   $(\sigma_{26}, 9) \oplus (\sigma_{12}, 5) \oplus (\mathbf{1}, 6)$  \\
    \hline
    $33$ & $27$ &  $(\tau_{12,9}, 7) \oplus (\mathbf{1}, 6)$   \\
    \hline
    $33$ & $32$ & $(\xi_{8,9}, 2) \oplus (\tau_{16,8}, 1) \oplus (\tau_{6,10}, 5)$\\
    \hline
    $35$ & $25$ &   $(\sigma_{26}, 11) \oplus (\sigma_{12}, 3) \oplus 
    (\mathbf{1}, 8)$  \\
    \hline
    $35$ & $31$ &   $ (\tau_{4,15}, 5) \oplus (\sigma_{20}, 7) \oplus (\sigma_{12}, 1)$   
    \end{tabular}
\end{table}

Theorem \ref{thm: main intro} now follows by 
substituting these results into Proposition \ref{prop: criterion}.  
More precise coverage is as follows. 

\begin{proof}[(Proof of Theorem \ref{thm: main intro})]
We begin with the assertion (i). 
The case $k \equiv 3 \bmod 4$ is covered by 
Proposition \ref{prop: input k odd}. 
In the case $k\equiv 1 \bmod 4$, 
the range $(n+23)/2 \leq k <n$ is covered by 
Proposition \ref{prop: input k odd}. 
The remaining range is covered by Proposition \ref{prop: k minimum} 
except for $(n, k)=(33, 25), (35, 25), (41, 29)$. 
The first two cases are covered by 
Proposition \ref{prop: input n small}. 
For the last case, we use 
the $A$-parameter 
$(\sigma_{30}, 13)\oplus (\sigma_{12}, 5)\oplus (\mathbf{1}, 6)$. 

The assertion (ii) is covered by 
Proposition \ref{prop: input k even} 
(and Proposition \ref{prop: input n small} for $(n, k)=(33, 32)$), 
and (iii) is covered by Propositions 
\ref{prop: k minimum} -- \ref{prop: input n small}.     
\end{proof}

The rest of this section is devoted to explaining and proving 
these propositions. 
In \S \ref{subsection_GL_2_autom_rep}, we recall $\sigma_k$ and $\Sym^2 \sigma_{k}$. 
In \S \ref{subsection_autom_rep_GL_4_5}, we explain $\tau_{j,k}$ and $\xi_{j,k}$. 
In \S \ref{ssec: input verify}, we carry out the proof of 
the above propositions in some sample cases. 
\S \ref{ssec: classify} is a supplement, 
where we give a classification result 
in the simplest case in Proposition \ref{prop: k minimum}.

\begin{remark}
For the convenience of the readers, 
we provide the following two additional tables: 

(1) 
Table \ref{table: input n=25, 27} 
is the list of $A$-parameters 
in $n=25, 27$ constructed in this paper.  

(2) 
Table \ref{table: middle n} 
is a list of degrees $k$ where we can construct 
holomorphic forms for $n=41, 43, 49, 51$ 
(without exhibiting the corresponding $A$-parameters). 
This extends Table \ref{table: small n intro}, 
and shows how the distribution of $k$ 
looks like. 
\end{remark}

\begin{remark}
    Canonical forms can be constructed for more general lattices. 
    This will be discussed elsewhere. 
\end{remark}

\renewcommand{\arraystretch}{1.2}
\begin{table}
    \centering
    \caption{$A$-parameters in $n=25, 27$}
    \label{table: input n=25, 27}
    \begin{tabular}{c|c|c}
    $n$& $k$ & \text{$A$-parameters} \\
    \hline 
   $25$ & $17$ & $(\sigma_{18}, 9) \oplus (\mathbf{1}, 8)$ \\ 
    \hline
    $25$ & $22$ & $(\Sym^2\sigma_{12}, 4) \oplus (\sigma_{12}, 7)$ \\ 
    \hline
    $27$ & $17$ & $(\sigma_{18}, 11) \oplus (\mathbf{1}, 6)$ \\ 
    \hline
    $27$ & $23$ & $(\tau_{8,9}, 5) \oplus (\sigma_{18}, 1) \oplus (\sigma_{16}, 1) \oplus (\mathbf{1}, 4)$   \\
    \hline
    $27$ & $25$ & $(\sigma_{26}, 3) \oplus (\sigma_{12}, 11)$  \\
    \hline
    $27$ & $26$ &  $(\xi_{4,10}, 2) \oplus (\tau_{4,10}, 3) \oplus (\sigma_{12}, 3)$   
\end{tabular}
\end{table}

\begin{table}
    \centering
    \caption{Semi-small $n$}
    \label{table: middle n}
    \begin{tabular}{c|c}
    $n$ & $k$ \\
    \hline
    $41$ & $21, 25, 29, 33, 35, 37, 38, 39, 40$  \\
    \hline
    $43$ & $25, 29, 33, 35, 37, 38, 39, 40, 41, 42$  \\
    \hline
    $49$ &  $25, 29, 33, 37, 39, 41, 43, 45, 46, 47, 48$  \\
    \hline
    $51$ &  $29, 33, 37, 39, 41, 43, 45, 46, 47, 48, 49, 50$  
    \end{tabular}
\end{table}

\subsection{Elliptic cusp forms}\label{subsection_GL_2_autom_rep}

Let $f$ be a Hecke eigenform of weight $k$ with 
respect to $\SL_2(\bbZ)$.
We denote by $\pi_{f}$ the cuspidal representation of $\GL_2(\bbA)$ generated by $f$. 
The archimedean component is 
the discrete series representation of $\GL_2(\bbR)$ 
with $L$-parameter $\rho_{k-1}$ 
(see, e.g., \cite{CL} p.166 and p.198). 
The representation $\pi_f$ is symplectic, as is well-known. 
In general, we use the notation $\sigma_k$ for 
the cuspidal representation $\pi_f$ of $\GL_2(\bbA)$ 
generated by an arbitrary Hecke eigenform $f$ 
of weight $k$ for $\SL_2(\bbZ)$. 

Next, for a given $\sigma_{k}$, 
Gelbart and Jacquet (\cite{GJ} Theorem 9.3) 
constructed its symmetric square lifting $\Sym^2\sigma_{k}$. 
This is a cuspidal representation of $\GL_3(\bbA)$ 
whose $L$-parameter is the symmetric square of that of 
$\sigma_k$. 
By a direct calculation, we see that 
the archimedean $L$-parameter of $\Sym^2\sigma_{k}$ 
is $\rho_{2k-2} \oplus \sgn$ and that   
$\Sym^2\sigma_{k}$ is orthogonal. 

\subsection{Siegel cusp forms of genus 2}\label{subsection_autom_rep_GL_4_5}

We denote by $S_{j, k}$ the space of Siegel cusp forms of weight 
$\Sym^j \otimes \det^k$ with respect to $\Sp_4(\bbZ)$. 
Note that $j$ must be even. 

\begin{lemma}\label{existence_SMF}
We have $S_{j, k}\ne 0$ for pairs $(j, k)$
described in Table \ref{table_existence_SMF}. 
\end{lemma}
\begin{proof}
    This is a consequence of \cite{Wak} Theorem 7.1 and \cite{Pet} p.45. 
    See also the table in \cite{vdG} \S 25. 
\end{proof}

\begin{table}
  \centering
  \caption{Range of $(j, k)$ with $S_{j, k}\ne0$}\label{table_existence_SMF}
  \begin{tabular}{c|cc}
    \hline
    $j \,\diagdown\, k$ & \textbf{even} $k$ & \textbf{odd} $k$\\
    \hline
    $0$ & $k \ge 10$ & $k=35$ or $k \ge 39$ \\
    $2$ & $k \ge 14$ & $k \ge 21$ \\
    $4$ & $k \ge 10$ & $k \ge 15$ \\
    $6$ & $k \ge 8$ & $k \ge 11$ \\
    $8$ & $k \geq 8$ & $k \ge 9$ \\
    $10$ & $k \geq 10$ & $k \geq 9$ \\
    $12$ & $k \geq 6$ & $k \geq 7$\\
    $14$ & $k \geq 8$ & $k \geq 7$\\
    $j \geq 16$ & $k \geq 6$ & $k \geq 7$\\
    \hline
    \multicolumn{3}{p{0.92\linewidth}}{\emph{Small $k$.}\; 
     If $k=3$, $j=36$ or $j\ge 42$. 
     If $k=4$,  $j=24$ or $j\geq 28$. 
     If $k=5$, $j=18,20$ or $j\ge 24$. 
     If $k=6$, $j=12$ or $j\ge 16$. 
     If $k=7$, $j\ge 12$. 
     } \\
    \hline
  \end{tabular}
\end{table}

Let $f$ be a Hecke eigenform in $S_{j, k}$. 
By the exceptional isomorphism $\SO(2,3) \isom \PGSp_4$, 
we may regard $f$ as an automorphic form on 
the split orthogonal group $\SO(2,3)(\bbA)$ 
(see, e.g., \cite{RS} \S A.7). 
Let $\pi_f = \otimes_{p \leq \infty} \pi_{f,p}$ be 
the cuspidal representation of 
$\SO(2,3)(\bbA)$ generated by $f$.
The archimedean component has $L$-parameter
$\rho_{2k+j-3} \oplus \rho_{j+1}$ 
(see, e.g., \cite{CL} p.166 and p.198). 

By the Arthur multiplicity formula for $\SO(2,3)$ (\cite{Arthur}), 
there exists a global $A$-parameter $\psi$ of $\SO(2,3)$ and 
a character $\eta = \prod_p\eta_p$ such that
$\pi(\psi_p, \eta_p) \isom \pi_{f,p}$ 
for every $p$.
When $\psi$ is of the form $ (\Pi, 1)$ 
for some cuspidal representation $\Pi$ 
of $\GL_4(\bbA)$, 
we call $\psi$ \textit{simple}. 
(See \cite{Sch} \S 1.1 for other possibilities of $\psi$, 
which we will not use.) 
By the condition in Definition \ref{def: A-packet}, 
$\Pi$ must be symplectic.

\begin{lemma}\label{lemma_vect_val_SMF_stable}
    Let $f$ be a Hecke eigenform in $S_{j, k}$. 
    If $j >0$, the corresponding $A$-parameter is simple.
    Moreover, if $k \geq 20$ and $S_{0, k} \neq 0$, 
    there exists a Hecke eigenform whose $A$-parameter is simple.
\end{lemma}
\begin{proof}
    The case $j>0$ is proved in \cite{Ish} Lemma 7.8.
    When $j = 0$, the space $S_{0, k}$ is spanned by 
    the Saito-Kurokawa lifting and Hecke eigenforms 
    corresponding to simple $A$-parameters 
    (\cite{RSY} Proposition 2.1 and \cite{Sch} Lemma 2.5). 
    Hence, it suffices to show that 
    when $k \geq 20$ and $S_{0,k} \neq 0$, 
    then $S_{0,k}$ is not spanned by the Saito-Kurokawa lifting.
    This follows from comparison of dimension with 
    the source of the Saito-Kurokawa lifting. 
\end{proof}

In general, we use the notation $\tau_{j, k}$ 
for the cuspidal representation of $\GL_4(\bbA)$ 
associated to an arbitrary Hecke eigenform in $S_{j, k}$ 
by the method of Lemma \ref{lemma_vect_val_SMF_stable}.

Next let $f$ be again a Hecke eigenform in $S_{j,k}$. 
We go back to viewing $f$ as an automorphic form on $\Sp_{4}(\bbA)$.
Note that the dual group of $\Sp_4$ is $\SO(3,2)\subset \GL_5$. 
By the Arthur multiplicity formula, now for $\Sp_4$ (\cite{Arthur}), 
we obtain the global $A$-parameter $\psi$ of $\Sp_4$ 
associated to $f$. 
As in the case of $\SO(2,3)$, 
if $f$ satisfies the conditions in Lemma \ref{lemma_vect_val_SMF_stable}, 
the $A$-parameter $\psi$ is of the form $(\Pi, 1)$ 
where $\Pi$ is a cuspidal representation of $\GL_5(\bbA)$ 
(rather than $\GL_4(\bbA)$).
By Arthur's multiplicity formula for $\Sp_4$, 
we see that $\Pi$ is orthogonal. 
The archimedean component of $\Pi$ has $L$-parameter 
$\rho_{2k+2j-2} \oplus \rho_{2k-4} \oplus \mathbf{1}$ 
(see \cite{Sch_arch} p.\ 2404 and \S 3.1). 
In general, we use the notation $\xi_{j, k}$ 
for the cuspidal representation of $\GL_5(\bbA)$ associated to an arbitrary Hecke eigenform in $S_{j,k}$ by this method. 

Our recollection in \S \ref{subsection_GL_2_autom_rep} 
and \S \ref{subsection_autom_rep_GL_4_5} 
can be summarized in Table \ref{table: building block}. 
There $(m, d)$ means the information $(m_i, d_i)$ in 
Definition \ref{def: A-packet}. 
The parity of $d$ corresponds to whether 
the automorphic representation is orthogonal or symplectic. 

\begin{table}
    \centering
    \caption{Building blocks}\label{table: building block}
    \begin{tabular}{c|c|c|c}
        Rep. & $L$-parameter at $\infty$ & $m$ & $d \bmod 2$  \\
        \hline
        $\mathbf{1}$ & $\mathbf{1}$ & $1$ & even \\ 
        \hline
        $\sigma_{k}$ & $\rho_{k-1}$ & $2$ & odd \\ 
        \hline
        $\Sym^2 \sigma_{k}$ & $\rho_{2k-2}\oplus \sgn$ & $3$ & even \\ 
        \hline
        $\tau_{j,k}$ & $\rho_{2k+j-3}\oplus \rho_{j+1}$ & $4$ & odd \\ 
        \hline 
        $\xi_{j,k}$ & $\rho_{2k+2j-2} \oplus \rho_{2k-4}\oplus \mathbf{1}$ 
        & $5$ & even 
    \end{tabular}
\end{table}

\subsection{Proof of propositions}\label{ssec: input verify}

In this subsection, we prove Propositions 
\ref{prop: k minimum}, \ref{prop: input k odd}, 
\ref{prop: input k even}, and \ref{prop: input n small} 
in some sample cases. 
Let us prepare some notation. 
In what follows, $(\rho_{k}, d)$ stands for 
the representation $\rho_k \boxtimes \Sym^{d-1}$ of the Weil group $W_{\R}$. 
For a given global $A$-parameter $\psi$, 
we use the notation 
\begin{equation*}
    \calS_{\psi} = 
    \bigoplus_{i} (\Z/2\Z) e_i \qquad 
    \calS_{\psi_{\infty}} = 
    \bigoplus_{j} (\Z/2\Z) e_{j,\infty} 
\end{equation*}
for the component groups of $\psi$ and $\psi_{\infty}$ 
with their standard generators corresponding to 
the given expression of the parameters. 
($\psi_{\infty}$ will be always Adams-Johnson.) 
Finally, 
recall that we always have $\eta_{\infty}(z_{\psi_{\infty}})=1$ 
(Proposition \ref{lemma_Adams_Johnson_packets}),  
which somewhat reduces the calculation of $\eta_{\infty}$.

\begin{proof}[(Proof of Proposition \ref{prop: k minimum})]
    That the $A$-parameters in Table \ref{table: simple parameter} 
    exist and satisfy the conditions 
    in Definition \ref{def: A-packet} can be checked immediately. 
    For example, in the case of $(\sigma_{(n+3)/2}, (n+1)/2)$, 
    we have $m_1d_1=2\cdot (n+1)/2=n+1$ and 
    $d_1=(n+1)/2$ is odd by our condition $n\equiv 1 \bmod 8$. 
    The bound $n\geq 33$ comes from the bound $(n+3)/2\geq 12, \ne 14$ 
    of weight of cusp forms for $\SL_2(\Z)$. 
    In the case of 
    $(\tau_{(n-3)/4, (n+9)/4}, (n+1)/4)$, 
    the bound $n\geq 35$ comes from Table \ref{table_existence_SMF}. 

The localizations of these $A$-parameters at $\infty$ are 
\begin{equation*}
(\rho_{(n+1)/2}, \: (n+1)/2), 
\end{equation*}
\begin{equation*}
(\rho_{(3n+3)/4}, (n+1)/4) \oplus (\rho_{(n+1)/4}, (n+1)/4), 
\end{equation*}
\begin{equation*}
(\rho_{(n+1+\ell)/2}, \, (n+1-\ell)/2) \oplus (\mathbf{1}, \ell).  
\end{equation*}
Clearly these satisfy the condition in Definition 
\ref{def: AJ} and the conditions 
\eqref{eqn: AJ equality} or \eqref{eqn: AJ equality II}. 
It remains to check $\vep_{\psi}=\eta_{\infty}\circ \Delta_{\infty}$. 
In fact, we have $\vep_{\psi}=\mathbf{1}$ in all cases. 
In the first and second case, this is obvious. 
In the last case, we have $\vep_{\psi}=\mathbf{1}$ because 
$\min (d_1, d_2)=\ell$ is even by our bound $\ell\leq (n+1)/3$. 

Finally, we check $\eta_{\infty}\circ \Delta_{\infty}=\mathbf{1}$. 
In the first and second case, we have 
$\eta_{\infty}\circ \Delta_{\infty}(e_1) = \eta(z_{\psi_{\infty}})=1$. 
In the last case, 
we have $d_0 = \ell \equiv 1-n \bmod 8$. 
By substituting this into \eqref{eqn: etaAJ}, 
we see that $\eta_{\infty}=\mathbf{1}$. 
\end{proof}

\begin{proof}[(Proof of Proposition \ref{prop: input k odd})]
    That the $A$-parameters in Table \ref{table: input odd k} 
    exist and satisfy the conditions in 
    Definition \ref{def: A-packet} can be checked 
    as before.  
    For example, in the case $(n, k)\equiv (1, 1)$, 
    the bound $d\leq (n-21)/2$ comes from the bound 
    $(n+1)/2-d \geq 12, \ne 14$ of weight of cusp forms for 
    $\SL_2(\Z)$ in the last component of the $A$-parameter. 

    The localizations of these $A$-parameters 
    at $\infty$ are calculated as 
    \begin{equation*}
        (\rho_{n+1-d}, \, d) \oplus (\rho_{n-2d}, \, 1) 
        \oplus (\rho_{(n-1)/2-d}, \, (n-1)/2-d) 
    \end{equation*}
    \begin{equation*}
        (\rho_{n+1-d}, \, d) \oplus 
        (\rho_{(n+1)/2-d}, (n+1)/2-d) 
    \end{equation*}
    \begin{equation*}    (\rho_{n-d+1}, \, d) \oplus 
        (\rho_{n-3d+1}, \, d) \oplus 
        (\rho_{(n+1)/2-2d}, \,(n+1)/2-2d)   
    \end{equation*}
    \begin{equation*}    (\rho_{n-d+1},  d) \oplus 
        (\rho_{n-3d+1},  d) \oplus 
        (\rho_{n-4d},  1) \oplus 
        (\rho_{(n-1)/2-2d},  (n-1)/2-2d).  
\end{equation*}
It is straightforward to verify the condition in 
Definition \ref{def: AJ} and \eqref{eqn: AJ equality}. 

Finally, we check $\vep_{\psi}=\eta_{\infty}\circ \Delta_{\infty}$. 
We have $\vep_{\psi}=\mathbf{1}$ in all cases 
because all indices $d_i$ are odd so that 
$\vep(\pi_i \times \pi_j)=1$ for any $i\ne j$. 
It remains to verify $\eta_{\infty}\circ \Delta_{\infty}=\mathbf{1}$. 
In the case $k\equiv 1$, we see that $\eta_{\infty}=\mathbf{1}$ 
by substituting $d_1\equiv d_2 \equiv 1$ (resp.~ $d_1\equiv 3$) mod $4$ 
into \eqref{eqn: etaAJ} in the case $n\equiv 1$ (resp.~$n\equiv 3$). 
Next we consider the case $(n, k)\equiv (1, 3)$. 
We have 
$\Delta_{\infty}(e_1) = e_{1,\infty} + e_{2,\infty}$ and 
$\Delta_{\infty}(e_2) = e_{3,\infty}$. 
By our condition, we have $d_1 = d_2 \equiv 3 \bmod 4$. 
Thus, by \eqref{eqn: etaAJ}, we have 
$\eta_{\infty}(e_{1,\infty}) = \eta_{\infty}(e_{2,\infty})=-1$ 
and $\eta_{\infty}(e_{3,\infty})=1$. 
Hence $\eta_{\infty}\circ \Delta_{\infty}=\mathbf{1}$ holds. 
The case $(n, k)\equiv (3, 3)$ is similar. 
\end{proof}

\begin{proof}[(Proof of Proposition \ref{prop: input k even})]
Verification of the conditions in Definition \ref{def: A-packet}, 
Definition \ref{def: AJ}, and 
\eqref{eqn: AJ equality} or \eqref{eqn: AJ equality II} is 
similar to the previous cases and left to the readers. 
We verify 
$\vep_{\psi}=\eta_\infty \circ \Delta_\infty$ 
in the cases $k\equiv 6$ and $(n, k)\equiv (1, 4)$ as samples. 
The second case is chosen because $\vep_{\psi}$ is nontrivial. 

We begin with the case $k\equiv 6$, where  
\[
\psi= (\Sym^2\sigma_{(n+3-d)/2}, \: d) \oplus 
(\sigma_{(n+3-d)/2}, \: (n+1-3d)/2), 
\]
\[
\psi_{\infty} = (\rho_{n+1-d}, \: d) \oplus (\rho_{(n+1-d)/2}, \: (n+1-3d)/2) 
\oplus (\sgn, \: d). 
\]
We have $\Delta_{\infty}(e_1)=e_{1,\infty}+e_{0,\infty}$ and 
$\Delta_{\infty}(e_2)=e_{2,\infty}$. 
We have $\vep_{\psi}=\mathbf{1}$ because 
$\min(d, (n+1-3d)/2)=d$ is even by our assumption 
$d\leq (n+1)/5$. 
On the other hand, 
calculating \eqref{eqn: etaAJ} with our condition $d\equiv n+3 \bmod 8$, 
we see that 
$\eta_{\infty}(e_{0,\infty}) = \eta_{\infty}(e_{1,\infty})$ 
and $\eta_{\infty}(e_{2,\infty})=1$. 
Hence $\eta_{\infty} \circ \Delta_{\infty}=\mathbf{1}$. 

Next we consider the case $(n, k)\equiv (1, 4)$, where 
\[ 
\psi = (\xi_{2d, (n-5d+3)/2}, d) \oplus (\sigma_{n-3d+1}, d+1) 
    \oplus \bigoplus_{m=1}^{\alpha} (\sigma_{n-6d+1-2m}, 1) \oplus (\sigma_{d+6}, 5)
    \]
     \[      
\psi_{\infty} =   
(\rho_{n-d+1}, d) \oplus (\rho_{n-3d}, d+1) 
    \oplus (\rho_{n-5d-1}, d) 
     \oplus \bigoplus_{m=1}^{\alpha} 
    (\rho_{n-6d-2m}, 1) \oplus (\rho_{d+5}, 5) \oplus (\mathbf{1}, d). 
    \]
Here we write $\alpha=(n-7d-11)/2$. 
The bound $d\leq (n-11)/7$ is required for the existence of 
the component $\oplus_{m=1}^{\alpha}$. 
Note that $\alpha\equiv 2 \bmod 4$. 
    We have 
    \[
    \Delta_\infty(e_{1}) = e_{1,\infty} + e_{3,\infty} + e_{0,\infty}, \quad  \Delta_\infty(e_{2}) = e_{2,\infty}, \quad 
    \Delta_\infty(e_{j}) = e_{j+1,\infty} \: \: (j>2). 
    \]
    Since $(n, d) \equiv (1, 6) \bmod 8$, 
    the character $\eta_\infty$ is calculated as
    \[
    \eta_{\infty}(e_{0,\infty})=1,\quad \eta_\infty(e_{1,\infty}) = 1, \quad  \eta_\infty(e_{2,\infty}) = 1, \quad  \eta_{\infty}(e_{3,\infty}) = -1
    \]
    and $\eta_{\infty}(e_{j,\infty}) = (-1)^{j}$ 
    for $j > 3$.
    Hence we have
    \begin{equation*}
\eta_\infty \circ \Delta_\infty(e_{i}) =
    \begin{cases}
        -1 & \quad i=1\\
        1 & \quad i=2\\
        (-1)^{i+1} & \quad i\geq 3.
    \end{cases}
    \end{equation*}

Next we calculate $\vep_\psi$. 
We need the following. 

\begin{claim}
We have 
$\vep(\xi_{2d, (n-5d+3)/2} \times \sigma_{k}) = (-1)^{k/2}$ 
for even $k$ with $k \leq n-5d$. 
\end{claim}

\begin{proof}
Following the procedure described in \cite{CL} p.203, 
we can calculate 
    \begin{align*}
    &\vep(\xi_{2d, (n-5d+3)/2} \times \sigma_{k}) = 
    \vep_\infty((\rho_{n-d+1} \oplus \rho_{n-5d-1} \oplus \mathbf{1}) \otimes \rho_{k-1})\\
    &=(-1)^{1+n-d+1} \cdot (-1)^{1+n-5d-1} \cdot (-1)^{k/2} 
    \; = \; (-1)^{k/2}. 
    \end{align*}    
\end{proof}

Let us abbreviate $\xi=\xi_{2d, (n-5d+3)/2}$. 
Then we can calculate $\vep_\psi(e_{2}) = 1$ (as before) and 
$$ 
\vep_{\psi}(e_{\textrm{last}}) = 
\vep(\xi \times \sigma_{d+6}) = (-1)^{(d+6)/2} = 1, 
$$
$$
\vep_{\psi}(e_{1}) = 
\prod_{m=1}^{\alpha}\vep(\xi \times \sigma_{n-6d+1-2m}) 
\cdot \vep(\xi \times \sigma_{d+6}) 
= \prod_{m=1}^{\alpha} (-1)^{1+m} = -1, 
$$
$$
\vep_\psi(e_{j}) = \vep(\xi \times \sigma_{n-6d+1-2j}) 
= (-1)^{(n-6d+1-2j)/2} = (-1)^{j+1}
$$
for $3\leq j < \textrm{last}$. 
This agrees with $\eta_\infty \circ \Delta_\infty$. 
\end{proof}

\begin{proof}[(Proof of Proposition \ref{prop: input n small})]
    This is similar to the previous cases. 
    Let us just verify 
    $\vep_{\psi}=\eta_{\infty}\circ \Delta_{\infty}$  
    in the case $(n, k)=(27, 23)$ as a sample,  
    where 
    \[
    \psi=(\tau_{8,9}, 5) \oplus (\sigma_{18}, 1) \oplus (\sigma_{16}, 1) \oplus (\mathbf{1}, 4), 
    \]
    \[
    \psi_{\infty} = 
    (\rho_{23}, 5) \oplus (\rho_{17}, 1) \oplus (\rho_{15}, 1) 
    \oplus (\rho_{9}, 5) \oplus (\mathbf{1}, 4).
    \]
    We have
    \[
    \Delta_{\infty}(e_1)=e_{1,\infty}+e_{4,\infty}, \quad 
    \Delta_{\infty}(e_2)=e_{2,\infty}, \quad 
    \Delta_{\infty}(e_3)=e_{3,\infty}, \quad 
    \Delta_{\infty}(e_0)=e_{0,\infty}. 
    \]
    The character $\vep_\psi$ is calculated as 
    \[
    \vep_{\psi}(e_i)=
    \begin{cases}
        1 & \quad i=1, 3 \\
        -1 & \quad i=0, 2 
    \end{cases}
    \] 
    by the parity of $d_i$ and 
    $\vep(\sigma_k \times \mathbf{1}) = (-1)^{k/2}$ 
    (see \cite{CL} p.203). 
    On the other hand, by calculating \eqref{eqn: etaAJ}, 
    we obtain  
    \[
    \eta_{\infty}(e_{i,\infty})=
    \begin{cases}
        1 & \quad i=3 \\
        -1 & \quad i=1, 2, 4, 0 
    \end{cases}
    \]
    From this we see that 
    $\eta_{\infty}\circ \Delta_{\infty}=\vep_{\psi}$ 
    holds. 
\end{proof}

\subsection{A classification}\label{ssec: classify}

Finally, we derive a converse result in the simplest case 
in Proposition \ref{prop: k minimum}.  

\begin{proposition}\label{cor: exact isom}
    Let $n \equiv 1 \bmod 8$. 
    The space of square-integrable forms in 
    $H^0(\calF_n, \Omega^{(n+1)/2})$ is isomorphic to 
    the space of cusp forms of weight $(n+3)/2$ for $\SL_2(\bbZ)$. 
\end{proposition}

\begin{proof}
By the argument in \S \ref{sec: auto rep to hol form},   
the dimension of the space of square-integrable forms in 
$H^0(\mathcal{F}_n, \Omega^{(n+1)/2})$ is equal to 
the multiplicity of $L(\wedge^{n-k}, -k)$ in 
$L^2_{disc}({\SO}^{+}(L)\backslash {\SO}^{+}(L\otimes \R))$. 
By the argument in \S \ref{sec: adelize}, 
we can pass from  
$$
L^2_{disc, \A} := L^2_{disc}({\SO}(L\otimes \Q)\backslash {\SO}(L\otimes \A)) 
$$ 
to $L^2_{disc}({\SO}^{+}(L)\backslash {\SO}^{+}(L\otimes \R))$ 
by taking the invariant part for $\prod_{p<\infty}{\SO}(L\otimes \Z_{p})$. 
Since ${\SO}(L\otimes \Z_{p})$ is hyperspecial, 
the ${\SO}(L\otimes \Z_{p})$-invariant part of 
an irreducible unramified representation of ${\SO}(L\otimes \Q_{p})$ 
has dimension $1$ (see \cite{GH} Corollary 5.6). 
Hence, by (the full version of) Theorem \ref{theorem_AMF} (\cite{2024_Ishimoto_Arthur}), 
we find that the desired multiplicity 
is equal to the number of $(\psi, \eta)$ contributing to 
$L^2_{disc, \A}$ such that 
$\pi(\psi, \eta)$ is unramified and 
$\pi(\psi_{\infty}, \eta_{\infty})\simeq 
L(\wedge^{n-k}, -k)$. 
By Remark \ref{remark: unramified local}, 
$\psi$ is unramified and $\eta_p=\mathbf{1}$ for $p<\infty$. 
By Lemma \ref{lem_lowest_degree}, 
we have 
\begin{equation}\label{eqn: psi infty special}
\psi_{\infty} = \rho_{(n+1)/2}\boxtimes {\rm Sym}^{(n-1)/2}. 
\end{equation}
Hence the desired multiplicity is equal to 
the number of global $A$-parameters $\psi$ 
which is unramified, satisfies \eqref{eqn: psi infty special}, 
and contributes to $L^2_{disc, \A}$. 
If we write $\psi=\oplus_i(\pi_i, d_i)$, 
then \eqref{eqn: psi infty special} shows that 
$d_i=(n+1)/2$ for some index $i$. 
The equality $\sum_i m_i d_i = n+1$ implies that 
there is no other index and $m_i=2$. 
Hence $\psi=(\pi_f, (n+1)/2)$ 
for an eigenform $f$ of weight $(n+3)/2$ for ${\rm SL}_2(\Z)$. 
Such $\psi$ indeed contributes to $L^2_{disc, \A}$ 
by Proposition \ref{prop: k minimum}, 
and with multiplicity $1$ by the Arthur's multiplicity formula. 
\end{proof} 

\appendix 

\section{Split orthogonal groups}\label{sec: everywhere split}

In this section, 
we classify rational quadratic forms of signature $(2, n)$ 
whose orthogonal group over ${\Qp}$ is split for any prime $p<\infty$. 
Since this plays only an auxiliary role in \S \ref{sec: lattice}, 
it is treated as an appendix. 

In what follows, the symbol $U$ stands for the hyperbolic plane over 
any given field of characteristic $0$ ($\Q$ or ${\Qp}$), 
namely the quadratic space expressed by the Gram matrix 
$\begin{pmatrix} 0 & 1 \\ 1 & 0 \end{pmatrix}$. 
(Note that this is different from the notation in 
\S \ref{sec: intro} and \S \ref{sec: lattice} where we worked integrally.) 
The discriminant of a quadratic space $V$ over a field 
$F$ is denoted by 
$\det V \in F^{\times}/(F^{\times})^2$.
All quadratic spaces are assumed to be nondegenerate. 

Let $V_{p}$ be a quadratic space over ${\Qp}$ for a prime $p<\infty$. 
The special orthogonal group ${\SO}(V_{p})$ 
is said to be \textit{split} 
if it has a maximal torus which splits over ${\Qp}$. 
It is well-known (see \cite{Bo} \S 23.4) that 
${\SO}(V_p)$ is split if and only if 
\begin{equation}\label{eqn: split quadratic form local}
V_{p} \; \simeq \; 
\begin{cases}
U \oplus \cdots \oplus U & \dim V_p \; \textrm{even} \\
U \oplus \cdots \oplus U \oplus \langle a \rangle & \dim V_p \; \textrm{odd} 
\end{cases}
\end{equation}
where $a \in {\Q}_{p}^{\times}$.

\begin{proposition}\label{prop: everywhere split}
Let $V$ be a rational quadratic space of signature $(2, n)$. 
The property 
\begin{equation}\label{eqn: split over any p}
{\SO}(V\otimes {\Qp}) \; \textrm{is split for every finite prime} \; p 
\end{equation}
holds if and only if $V$ is isometric to one of the following quadratic spaces: 
\begin{equation}\label{eqn: split quadratic form list}
\begin{cases}
2 \langle 1 \rangle  \oplus (8m+2) \langle -1 \rangle & n\equiv 2 \; {\rm mod} \; 8 \\
2 \langle 1 \rangle  \oplus (8m+2) \langle -1 \rangle \oplus \langle -d \rangle & n\equiv 3 \; {\rm mod} \; 8  \\
2 \langle 1 \rangle  \oplus (8m-6) \langle -1 \rangle \oplus K_{d} & n\equiv 1 \; {\rm mod} \; 8 
\end{cases} 
\end{equation} 
Here $d>0$ is some natural number and  
$K_{d}$ is the orthogonal complement of an embedding 
$\langle -d \rangle\hookrightarrow 8 \langle -1 \rangle$. 
\end{proposition}

We can also extend this classification from 
split to quasi-split (when $n$ is even), 
but since this is longer with less relevance to other parts of the paper, 
we omit it. 

For the proof we need to recall the Hasse invariants. 
We refer to \cite{Ge} Chapter 4 for what follows. 
Let $V_{p}$ be a quadratic space over ${\Qp}$ for a prime $p< \infty$. 
We choose an orthogonal basis 
$V_{p} \simeq \langle a_1, \cdots, a_{N} \rangle$ 
where $a_{i}\in {\Q}_{p}^{\times}$. 
When $\dim V_p >1$, 
the Hasse invariant of $V_{p}$ is defined by 
\begin{equation}\label{eqn: Hasse invariant}
\vep_p(V_p) = \prod_{i<j} (a_i, a_j)_{p} \; \; \in \{ \pm 1\},  
\end{equation}
where $(a, b)_{p}\in \{ \pm 1\}$ is the Hilbert symbol. 
This is independent of the choice of orthogonal basis. 
We set $\vep_p (V_p)=1$ when $\dim V_p=1$. 
The isometry class of $V_p$ is determined by the triplet 
\begin{equation*}
(\dim V_p, \: \vep_p (V_p), \; \det V_p). 
\end{equation*}
The Hasse invariant satisfies the product formula 
\begin{equation*}\label{eqn: Hasse product}
\vep_p(V_p \oplus W_p) \; = \; 
\vep_p (V_p) \cdot \vep_p (W_p) \cdot (\det V_p, \det W_p)_{p}. 
\end{equation*}
The special case 
\begin{equation}\label{eqn: Hasse mU}
\vep_p(kU) = (-1, -1)_{p}^{k(k-1)/2}, \quad k>0, 
\end{equation}
will be used repeatedly in the following. 

The Hasse invariant $\vep_{\infty}(V_{\infty})$ 
of a real quadratic space $V_{\infty}$ is defined similarly by \eqref{eqn: Hasse invariant}, 
where $(a, b)_{\infty}=-1$ precisely when both $a$ and $b$ are negative.  
When $V_{\infty}$ has signature $(\ast, n)$, we have 
$\vep_{\infty}(V_{\infty})=(-1)^{n(n-1)/2}$. 

Let $V$ be a quadratic space over ${\Q}$. 
We write $\vep_p(V)= \vep_p(V\otimes {\Qp})$. 
We have $\vep_p(V)\ne 1$ for only finitely many $p$. 
Then the Hilbert reciprocity 
\begin{equation*}\label{eqn: reciprocity}
\prod_{p\leq \infty} \vep_p (V) =1 
\end{equation*}
holds. 

We can now give the proof of Proposition \ref{prop: everywhere split}. 

\begin{proof}[(Proof of Proposition \ref{prop: everywhere split})]
The proof is divided into several steps. 

\begin{step}
The quadratic spaces \eqref{eqn: split quadratic form list} satisfy \eqref{eqn: split over any p}. 
\end{step}

\begin{proof}
We have  
$8 \langle -1 \rangle \simeq 4U$ over ${\Q}_{p}$ for every $p<\infty$. 
Then $K_d\otimes {\Qp} \simeq 3U \oplus \langle d \rangle$ 
by the Witt cancellation. 
This proves our assertion. 
\end{proof}

The rest of the proof is devoted to verifying the ``only if" direction. 
We first consider the case $n$ even. 

\begin{step}
When $n\equiv 2$ mod $8$, 
the quadratic space $2 \langle 1 \rangle  \oplus (8m+2) \langle -1 \rangle$ 
is the only one satisfying \eqref{eqn: split over any p}. 
\end{step}

\begin{proof}
The condition \eqref{eqn: split quadratic form local} determines $V\otimes {\Qp}$ for every $p<\infty$,  
while the signature condition determines $V\otimes {\R}$. 
Then $V$ is uniquely determined by the Hasse-Minkowski theorem. 
\end{proof}

\begin{step}
When $n\equiv 0$ mod $4$, there is no quadratic space $V$ satisfying \eqref{eqn: split over any p}. 
\end{step}

\begin{proof}
By the condition \eqref{eqn: split quadratic form local}, 
$V\otimes {\Qp}$ is a direct sum of odd numbers of $U$, 
so we have $\det (V\otimes {\Qp})=-1$ in ${\Q}_{p}^{\times}/({\Q}_{p}^{\times})^2$. 
Since the natural map 
\begin{equation*}
{\Q}^{\times}/({\Q}^{\times})^2 \to \prod_{p<\infty} {\Q}_{p}^{\times}/({\Q}_{p}^{\times})^2 
\end{equation*}
is injective, we have $\det V=-1$ in ${\Q}^{\times}/({\Q}^{\times})^2$. 
However, since $V\otimes {\R}$ has signature $(2, n)$, $\det V$ must be positive. 
This is absurd. 
\end{proof}

\begin{step}
When $n\equiv 6$ mod $8$, there is no quadratic space $V$ 
satisfying \eqref{eqn: split over any p}. 
\end{step}

\begin{proof}
In this case, $V\otimes {\Qp}$ is a direct sum of copies of $4U$, 
so we have $\vep_p(V)=1$ for any $p<\infty$ by \eqref{eqn: Hasse mU}. 
On the other hand, since $V$ has signature $(2, n)$ with 
$n\equiv 2$ mod $4$, we have 
$\vep_{\infty}(V)= (-1)^{n(n-1)/2}=-1$. 
This violates the Hilbert reciprocity. 
\end{proof}

The proof in the case of even $n$ is now finished. 
Next we consider the case when $n$ is odd.

\begin{step}
Let $n$ (odd) and $d\in {\Q}^{\times}/({\Q}^{\times})^2$ be fixed. 
A rational quadratic space $V$ of signature $(2, n)$ and discriminant $d$ 
satisfying \eqref{eqn: split over any p} 
is unique if it exists.  
\end{step}

This assures in the case $n\equiv 1, 3$ mod $8$ that  
the quadratic spaces \eqref{eqn: split quadratic form list} are the only ones satisfying \eqref{eqn: split over any p} 
because they exhaust all possible discriminants. 

\begin{proof}
If we write 
$V\otimes {\Qp} \simeq U \oplus \cdots \oplus U \oplus \langle a \rangle$, 
then $a$ is determined by $d$ and $n$. 
Therefore $V\otimes {\Qp}$ is uniquely determined for every $p<\infty$. 
Since the signature $(2, n)$ is fixed, 
$V$ is uniquely determined by the Hasse-Minkowski theorem. 
\end{proof}

It remains to prove the non-existence in the case $n\equiv 5, 7$ mod $8$. 

\begin{step}\label{step6}
When $n\equiv 5, 7$ mod $8$, there is no quadratic space $V$ satisfying \eqref{eqn: split over any p}. 
\end{step}

\begin{proof}
We have 
\begin{equation}\label{eqn: einftyV}
    \vep_{\infty}(V) = (-1)^{n(n-1)/2} = 
    \begin{cases}
    1 & n \equiv 5 \bmod 8 \\ 
    -1 & n \equiv 7 \bmod 8.
    \end{cases}
    \end{equation}
On the other hand, for $p<\infty$, 
the condition \eqref{eqn: split quadratic form local} says that 
$V\otimes {\Qp} \simeq (4m+\alpha)U\oplus \langle d \rangle$ 
where $\alpha=3$ or $0$ according to $n\equiv 5$ or $7$ mod $8$, 
and $d=\pm \det V$. 
In the case $n\equiv 7 \bmod 8$, 
we can see from the product formula and \eqref{eqn: Hasse mU} that 
$\vep_{p}(V)=\vep_{p}(4mU)=1$ 
for $p<\infty$. 
This violates the Hilbert reciprocity. 

In the case $n\equiv 5 \bmod 8$, 
we have $d=-\det V$, and hence 
\begin{equation*}
\vep_p(V) =  
\vep_p((4m+3)U) \cdot (-1, -\det V)_{p} = 
(-1, \det V)_p 
\end{equation*}
for every $p<\infty$ 
by the product formula and \eqref{eqn: Hasse mU}. 
Hence we have 
$\vep_{\infty}(V)=(-1, \det V)_{\infty}$ 
by the Hilbert reciprocity. 
Since $n$ is odd, $\det V$ is negative, 
and so 
$\vep_{\infty}(V)=-1$. 
This contradicts with \eqref{eqn: einftyV}. 
\end{proof}

The proof of Proposition \ref{prop: everywhere split} is now complete. 
\end{proof}

\section{Kodaira dimension}\label{sec: Kodaira dim}

It is proved by Gritsenko-Hulek-Sankaran \cite{GHS} that  
$\mathcal{F}_n$ with $n\equiv 3 \bmod 8$ is of general type if $n\geq 43$. 
In this section, we consider the case $n\equiv 1 \bmod 8$ 
and prove the following analogous result. 
Since this is independent of (though related to) 
other part of this paper, it is treated as an appendix. 

\begin{proposition}\label{prop: Kodaira dim}
Let $n\equiv 1 \bmod 8$. 
Then $\mathcal{F}_n$ is of general type if $n=25$ or $n\geq 41$. 
\end{proposition}

\begin{proof}
    In the case $n=25$, we take the quasi-pullback of 
    the Borcherds $\Phi_{12}$ form by 
    $\langle 2 \rangle \oplus U \oplus 3E_8 \hookrightarrow 2U \oplus 3E_8$. 
    This produces a cusp form of weight $12+1=13$ for 
    $\langle 2 \rangle \oplus U \oplus 3E_8$, 
    with which we can use the method of \cite{GHS07}. 

    In the case $n\geq 41$, 
    we use the method of \cite{GHS}, \cite{Ma1}. 
    In what follows, we use the notation in \cite{Ma1}. 
    Let $L=\langle 2 \rangle \oplus U \oplus mE_8$. 
    First, by a computation similar to \cite{Ma1} \S 7.2, 
    we see that there exists a cusp form of weight $13/2$ 
    for ${\Mp}_2(\Z)$ with values in the Weil representation for $L$. 
    Taking the Gritsenko lifting, we obtain a cusp form of weight 
    $13/2+n/2-1 = n/2+11/2$ for ${\rm O}^{+}(L)$. 
    Hence, by the argument in \cite{Ma1} \S 1.1, 
    we see that $\mathcal{F}_n$ is of general type if 
    $a\mathcal{L}-B/2$ is big, where 
    $$
    a=n-(n/2+11/2) = n/2-11/2 = 4m-5. 
    $$
    The branch divisor $B$ is irreducible and defined by 
    the sublattice 
    $$
    K=\langle 2 \rangle \oplus \langle 2 \rangle \oplus m E_8 \simeq 
    2U\oplus (m-1)E_8 \oplus D_6
    $$ 
    of $L$. 
    Then, by the proof of \cite{Ma1} Proposition 4.3, 
    $a\mathcal{L}-B/2$ is big if the inequality 
    \begin{equation}\label{eqn: volume inequality}
      2 \cdot \frac{{\rm vol}_{HM}({\rm O}^{+}(K))}{{\rm vol}_{HM}({\rm O}^{+}(L))} 
      \; < \; 
      \left( 1+ \frac{1}{a} \right)^{-8m} \frac{2a}{8m+1}
    \end{equation}
    holds. 
    It is calculated in \cite{GHS} p.12 that 
    $$
    {\rm vol}_{HM}({\rm O}^{+}(L)) = 
    2^{-4m} \cdot \prod_{k=1}^{4m+1} \frac{|B_{2k}|}{2k},  
    $$
    where $B_{2k}$ are the Bernoulli numbers. 
    Similarly, it is calculated (implicitly) in \cite{Ma1} \S 7.2 that 
    $$
    {\rm vol}_{HM}({\rm O}^{+}(K)) = 
    \pi^{-4m-1} \cdot (4m)!  \cdot 
    \prod_{k=1}^{4m} \frac{|B_{2k}|}{2k} \cdot L(4m+1, \chi_{-4}), 
    $$
    where $L(s, \chi_{-4})$ is the Dirichlet $L$-function 
    for the Kronecker symbol $\chi_{-4} = \left( \frac{-4}{\cdot} \right)$. 
    Substituting these two formulas, 
    we see that \eqref{eqn: volume inequality} holds when $m\geq 5$. 
\end{proof}


\begin{thebibliography}{99}


\bibitem{AJ} 
Adams, J.; Johnson, J. F. 
\textit{Endoscopic groups and packets of non-tempered representations.} 
Compos. Math. \textbf{64} (1987), 271--309, 1987. 

\bibitem{AMR} 
Arancibia, N.; Moeglin, C.; Renard, D.  
\textit{Arthur packets for classical and unitary groups.} 
Ann. Fac. Sci. Toulouse Math. \textbf{27} no.5 (2018), 1023--1105. 

\bibitem{Arthur} 
Arthur, J.  
\textit{The endoscopic classification of representations.} 
AMS, 2013. 

\bibitem{AGI+}
Atobe, H.; Gan, W. T.; Ichino, A.; Kaletha, T.; M\'inguez, M.; Shin, S. W. 
\textit{Local intertwining relations and co-tempered A-packets of 
classical groups.} arXiv:2410.13504. 


\bibitem{Atobe_A_packets}
Atobe, H. 
\textit{Moeglin's explicit construction of local A-packets.} 
in ``On the Langlands program: endoscopy and beyond.'' 141--207, 
World Scientific, 2024.


\bibitem{Bo} 
Borel, A. 
\textit{Linear algebraic groups.} 2nd ed. 
GTM \textbf{126}, Springer, 1991. 


\bibitem{CL}
Chenevier, G.; Lannes, J. 
\textit{Automorphic forms and even unimodular lattices.} 
Springer, 2019. 

\bibitem{CR} 
Chenevier, G.; Renard, D. 
\textit{Level one algebraic cusp forms of classical groups of small rank.} 
Mem. Amer. Math. Soc. \textbf{1121}, 2015. 

\bibitem{EHW_83}
Enright, T.; Howe, R.; Wallach, N. 
\textit{A classification of unitary highest weight modules.} 
in ``Representation theory of reductive groups (Park City, Utah, 1982)'', 
97--143, Birkh\"auser, 1983

\bibitem{Fr}
Freitag, E. 
\textit{Holomorphe Differentialformen zu Kongruenzgruppen der Siegelsche Modulgruppe.} 
Invent. Math. \textbf{30} (1975), 181--196. 

\bibitem{vdG}van der Geer, G. 
\textit{Siegel modular forms and their applications.} 
The 1-2-3 of modular forms, 181--245, Springer, 2008. 


\bibitem{Ge}
Gersten, L.  
\textit{Basic quadratic forms.} 
GSM \textbf{90}, AMS, 2008. 

\bibitem{GJ}
Gelbart, S.; Jacquet, H. 
\textit{A relation between automorphic representations of $GL(2)$ 
and $GL(3)$.} 
Ann. Sci. \'Ecole Norm. Sup. (4) \textbf{11} (1978), no.4, 471--542. 

\bibitem{GH}
Getz, J. R.; Hahn, H. 
\textit{An introduction to automorphic representations.} 
GTM \textbf{300}, Springer, 2024. 

\bibitem{GHS07}
Gritsenko, V.; Hulek, K.; Sankaran, G. K.
\textit{The Kodaira dimension of the moduli of $K3$ surfaces.} 
Invent. Math. \textbf{169} (2007), 519--567. 

\bibitem{GHS}
Gritsenko, V.; Hulek, K.; Sankaran, G. K.
\textit{Hirzebruch-Mumford proportionality and 
locally symmetric varieties of orthogonal type.} 
Doc. Math. \textbf{13} (2008), 1--19. 

\bibitem{HMY}
Horinaga, S.; Maeda, Y.; Yamauchi, T. 
\textit{The Kodaira dimension of even-dimensional ball quotients.} 
arXiv:2507.22203. 
  
\bibitem{Hu}
Humphreys, J.~E. 
\textit{Representations of semisimple Lie algebras in the BGG category $\mathcal{O}$.} 
GSM \textbf{94}, AMS, 2008.  


\bibitem{2024_Ishimoto_Arthur}
Ishimoto, H. 
\textit{The endoscopic classification of representations of non-quasi-split odd special orthogonal groups.} 
Int. Math. Res. Not. \textbf{2024}, no.14, 10939--11012.

\bibitem{Ish} 
Ishimoto, H.  
\textit{Ibukiyama correspondences on automorphic forms on 
${\rm Mp}_4({\A}_{\Q})$ and ${\SO}_5({\A}_{\Q})$ 
generating large discrete series representations at the real place.} 
J. Number Theory, \textbf{277} (2025), 63--104.

\bibitem{Kaletha-Prasad_Bruhat-Tits}
Kaletha, T.; Prasad, G. 
\textit{Bruhat-Tits theory -- a new approach.} 
Cambridge Univ. Press, 2023. 

\bibitem{Ma0}
Ma, S. 
\textit{Rationality of the moduli spaces of 2-elementary $K3$ surfaces.} 
J. Algebraic Geom. \textbf{24} (2015), no.1, 81--158. 

\bibitem{Ma1}
Ma, S. 
\textit{On the Kodaira dimension of orthogonal modular varieties.} 
Invent. Math. \textbf{212} (2018), no. 3, 859--911.

\bibitem{Ma2}
Ma, S. 
\textit{Vector-valued orthogonal modular forms.} 
Mem. Eur. Math. Soc. \textbf{21}, 2025.  

\bibitem{Milne}
Milne, J. S. 
\textit{Algebraic groups.} 
Cambridge Univ. Press, 2017. 

\bibitem{MR} Moeglin, C.; Renard, D. 
\textit{On Arthur packages of classical real groups.} 
J. Eur. Math. Soc. \textbf{22} (2020), no.6, 1827--1892. 

\bibitem{Pet} 
Petersen, D. 
\textit{Cohomology of local systems on 
the moduli of principally polarized abelian surfaces.} 
Pac. J. Math. \textbf{275} (2015) no.1, 39--61. 

\bibitem{Po}Pommerening, K.  
\textit{Die Fortsetzbarkeit von Differentialformen auf arithmetischen Quotienten von hermiteschen symmetrischen R\"aumen.} 
J. Reine Angew. Math. \textbf{356} (1985), 194--220. 

\bibitem{RS}
Roberts, B.; Schmidt, R. 
\textit{Local newforms for $\mathrm{GSp}(4)$}.
Lect. Notes Math. \textbf{1918}, Springer, 2007.

\bibitem{RSY}
Roy, M.; Schmidt, R.; Yi, S. 
\textit{On counting cuspidal automorphic representations for $\mathrm{GSp}(4)$.} 
Forum Math. \textbf{33} (2021), no.3, 821--843.

\bibitem{Sch_arch}
Schmidt, R.
\textit{Archimedean aspects of Siegel modular forms of degree $2$.} 
Rocky Mt. J. Math. \textbf{47} (2017), no.7, 2381--2422.

\bibitem{Sch} 
Schmidt, R. 
\textit{Packet structure and paramodular forms.} 
Trans. Amer. Math. Soc. \textbf{370} (2018), no.5, 3085--3112. 

\bibitem{taibi}
Ta\"ibi, O. 
\textit{Dimensions of spaces of level one automorphic forms for split classical groups using the trace formula.} 
Ann. Sci. \'Ec. Norm. Sup\'er. (4) \textbf{50} (2017), 
no.2, 269--344. 

\bibitem{Taibi_AMF}
Ta\"ibi, O. 
\textit{Arthur's multiplicity formula for certain inner forms of special orthogonal and symplectic groups.} 
J. Eur. Math. Soc. \textbf{21} (2019), no.3, 839--871.

\bibitem{VZ}
Vogan, D. A.; Zuckerman, G. J. 
\textit{Unitary representations with nonzero cohomology.} 
Compos. Math., \textbf{53} (1984), 51--90. 

\bibitem{Wak}
Wakatsuki, S. 
\textit{Dimension formulas for spaces of vector-valued 
Siegel cusp forms of degree two.} 
J. Number Theory \textbf{132} (2012), no.1, 200--253. 

\bibitem{Wal}
 Wallach, N. R. 
 \textit{On the constant term of a square integrable automorphic form.} 
in ``Operator algebras and group representations, Vol. II (Neptun, 1980)", 
227--237, 1984. 

\bibitem{We}
Weissauer, R. 
\textit{Vektorwertige Siegelsche Modulformen kleinen Gewichtes.} 
J. Reine Angew. Math. \textbf{343} (1983), 184--202. 

\end{thebibliography}
\end{document}